\documentclass[12pt]{amsart}
\usepackage{latexsym}
\usepackage{amssymb, amsmath,mathrsfs}
\usepackage[left=2.5cm,right=2.5cm,top=2.5cm, bottom=2.5cm ]{geometry}
\usepackage[OT2,T1]{fontenc}

\usepackage{setspace}
\usepackage[pagebackref,hypertexnames=false, colorlinks, citecolor=red, linkcolor=red]{hyperref}

\newcommand{\M}{\mathcal{M}}
\newcommand{\N}{\mathcal{N}}

\newcommand{\Smax}{\mathcal{S}^{max}}
\newcommand{\Smin}{\mathcal{S}^{min}}

\newtheorem{theorem}{Theorem}[section]
\newtheorem{lemma}[theorem]{Lemma}

\newtheorem{remark}[theorem]{Remark}

\newtheorem{corollary}[theorem]{Corollary}
\newtheorem{proposition}[theorem]{Proposition}

\DeclareSymbolFont{cyrletters}{OT2}{wncyr}{m}{n}
\DeclareMathSymbol{\Sha}{\mathalpha}{cyrletters}{"58}
\author[Bickel]{Kelly Bickel$^\dagger$}
\address{Department of Mathematics, Bucknell University, 360 Olin Science Building, Lewisburg, PA 17837, USA.}
\email{kelly.bickel@bucknell.edu}
\thanks{$\dagger$ Research supported in part by National Science Foundation
DMS grant \#1448846.}

\author[Liaw]{Constanze Liaw}
\address{Department of Mathematics, Baylor University, One Bear Place \#97328, Waco, TX 76798, USA.}
\email{Constanze$\underline{\,\,\,}$Liaw@baylor.edu}

\keywords{model spaces, two complex variables, compressed shift, Agler decomposition, essential normality}
 \subjclass[2010]{47A13, 47A20, 46E22}

\begin{document}

\title{Properties of Beurling-Type Submodules via Agler Decompositions}
\date{\today}

\maketitle
\begin{abstract}
In this paper, we study operator-theoretic properties of the compressed shift operators $S_{z_1}$ and $S_{z_2}$ on complements of submodules of the Hardy space over the bidisk $H^2(\mathbb{D}^2)$. Specifically, we study Beurling-type submodules -- namely submodules of the form $\theta H^2(\mathbb{D}^2)$ for $\theta$ inner -- using properties of Agler decompositions of $\theta$ to deduce properties of $S_{z_1}$ and $S_{z_2}$ on  model spaces $H^2(\mathbb{D}^2) \ominus \theta H^2(\mathbb{D}^2)$. Results include characterizations (in terms of $\theta$) of when a commutator $[S_{z_j}^*, S_{z_j}]$ has rank $n$ and when subspaces associated to Agler decompositions are reducing for $S_{z_1}$ and $S_{z_2}$.
We include several open questions.
\end{abstract}

\section{Introduction}
\subsection{Motivation} The Hardy space on the disk $H^2(\mathbb{D})$ has played a prominent role in developing both function and operator theory over the past century. Of particular importance are its shift-invariant subspaces, which (as proved by Beurling in \cite{b48}) are always of the form $\theta H^2(\mathbb{D})$ 
for an inner function $\theta.$ Indeed, the model theory of Sz.-Nagy-Foias \cite{SzNF2010} shows that every completely non unitary contraction is unitarily equivalent to the compression of multiplication by $z$ on some $\mathcal{K}_{\theta} \equiv H^2(\mathbb{D}) \ominus \theta H^2(\mathbb{D})$, as long as $\theta$ can be operator-valued. 

We are interested in generalizations of one variable Hardy space theory to the Hardy space over the bidisk $H^2(\mathbb{D}^2).$ Substantial progress in this direction has been made by W. Rudin, R.G.~Douglas, M. Gadadhar, R. Yang and many others \cite{dr00, dg03, dg08, int04, Rud69, Y01, Y02}, who often frame the important problems in terms of Hilbert submodules. In our situation, a Hilbert submodule $M$ in $H^2(\mathbb{D}^2)$ is a subspace that is invariant under 
 multiplication by the two independent variables $z_1$ and $z_2$, namely invariant under the Toeplitz operators $T_{z_1}, T_{z_2}$ \cite{dr00}.  We are interested in Beurling-type submodules, namely those of the form:
\[ M \equiv \theta H^2(\mathbb{D}^2),\]
where $\theta$ is inner. As shown by Mandrekar in \cite{M88}, these submodules are exactly the ones on which $T_{z_1}$ and $T_{z_2}$ are doubly commuting. In analogy with one variable model theory, given $\mathcal{K}_{\theta} \equiv H^2(\mathbb{D}^2) \ominus \theta H^2(\mathbb{D}^2),$ we are particularly interested in the compressed shift operators:
\[ 
S_{z_1} \equiv P_{\theta} T_{z_1}|_{\mathcal{K}_{\theta}} \ \text{ and } \  S_{z_2} \equiv P_{\theta} T_{z_2}|_{\mathcal{K}_{\theta}}, 
\]
where $P_{\theta}$ denotes the projection onto $\mathcal{K}_{\theta}$ and $\theta$ is inner. The case of general analytic contractions $\theta$ is quite involved even when we consider functions of only one complex variable. See for example \cite{RonRkN, LT}, which concerns Clark theory in the general situation, and the references therein. 

The literature already contains a variety of results concerning commutators of $S_{z_1}$, $S_{z_2}$ and their adjoints, as these operators are crucially related to both $\theta$ and the structure of $\mathcal{K}_{\theta}.$ For example, \cite{dr00,ii06,  int04, Y01, Y02} contain interesting results concerning the behaviors of the commutators
\[ [S_{z_1}, S^*_{z_2}] \text{ and } [S_{z_1}, S_{z_2} ].\]
However, the independent behavior of $S_{z_1}$ or $S_{z_2}$ is not completely understood. Some results exist concerning the essential spectrum of these operators under additional conditions \cite{Y02}, but in general, their operator theoretic properties and connections to $\theta$ are still mysterious.  The deepest such result (known by the authors) is due to Guo-Wang \cite{gw09}, which states:  $S_{z_1}$ and $S_{z_2}$ are both essentially normal, (i.e.~the commutators $[S^*_{z_1}, S_{z_1} ]$ and $[S_{z_2}^*, S_{z_2}]$
are both compact) if and only if $\theta$ is a rational inner function of degree at most $(1,1).$  Here, we say that the degree of $\theta $ is $(m_1,m_2)$, if we can write $\theta=p/q$ with polynomials $p$ and $q$ that share no common factor where $m_j$ is the maximum degree of $p$ and $q$ in $z_j$ for $j=1,2.$  In this paper, we study the individual behavior of $S_{z_1}$ and $S_{z_2}$ on $\mathcal{K}_{\theta}$ and in doing so, obtain a result distinct from, but complementary to, the Guo-Wang theorem. 

\subsection{Main Idea} Our method of approach is the following:~we disentangle the separate behaviors of $S_{z_1}$ and $S_{z_2}$ using canonical decompositions of $\mathcal{K}_{\theta}$ into $z_1$ and $z_2$ invariant subspaces. The existence of such decompositions follows immediately from the existence of \emph{Agler decompositions}. Specifically, in 1990 \cite{ag90}, J. Agler showed that every analytic contraction $\theta$
on the bidisk can be decomposed using two positive kernels $K_1, K_2: \mathbb{D}^2 \times \mathbb{D}^2 \rightarrow \mathbb{C}$ as follows:
\[ 1- \theta(z) \overline{\theta(w)} = (1-z_1 \bar{w}_1) K_2(z,w) + (1-z_2 \bar{w}_2)K_1(z,w). \]
 In \cite{ag1}, Agler used these kernels to generalize the classic Pick Interpolation Theorem to two variable and in the interim, this kernel formula has been used frequently to both generalize one variable results and address strictly multivariate questions on the polydisk as in \cite{agmc_isb, agmc_dv, mcc10a, amy10a, amy12, baltre98, kn07b}. 

In this paper, we study the connection between Agler kernels of $\theta$ and the operators $S_{z_1}, S_{z_2}$  on $\mathcal{K}_{\theta}$. Indeed, the question driving the majority of this paper is:
\begin{center} \emph{What do the Agler decompositions of $\theta$ imply about the operators $S_{z_1}$ and $S_{z_2}$ on $\mathcal{K}_{\theta}$?} \end{center}
Notice that the formula defining Agler decompositions can be rewritten as follows:
\begin{equation}
\label{eqn:akernels}
\frac{1-\theta(z) \overline{\theta(w)}}{(1-z_1 \bar{w}_1)(1-z_2 \bar{w}_2)} = \frac{K_1(z,w)}{1-z_1\bar{w}_1} + \frac{K_2(z,w)}{1-z_2 \bar{w}_2},
\end{equation} 
which is equivalent to a decomposition of $\mathcal{K}_{\theta} \equiv H^2(\mathbb{D}^2) \ominus \theta H^2(\mathbb{D}^2)$ into a $z_1$-invariant Hilbert space and a $z_2$-invariant Hilbert space; $\mathcal H\left(\frac{K_1(z,w)}{1-z_1\bar{w}_1}\right)$ and $\mathcal H\left(\frac{K_2(z,w)}{1-z_2\bar{w}_2}\right)$, respectively, where $\mathcal H(K)$ denotes the reproducing kernel Hilbert space with reproducing kernel $K$. We call these spaces \emph{Agler subspaces of $\theta$} and the pair $(K_1, K_2)$ \emph{Agler kernels of $\theta$}. Although these kernels (Hilbert spaces) are rarely unique, each $\theta$ does possess two canonical decompositions \cite{bk13}.

%Before stating these decompositions, recall that
%\[ L^2_{--} = \{ f \in L^2(\mathbb{T}^2): \widehat{f}(n_1,n_2) =0 \text{ whenever } n_1 > 0 \text{ or } n_2 >0 \}.\]
%Now, define the following subspaces of $H^2(\mathbb{D}^2)$ associated to $\theta$:
%\[
%\begin{aligned}
%\mathcal{R}_{\theta} &=  \theta L^2_{--}  \cap H^2(\mathbb{D}^2) \\
%\mathcal{H}_{\theta}^j &= z_j \theta L^2_{--}  \cap H^2(\mathbb{D}^2)  \qquad j=1,2\\
%\mathcal{M}_{\theta}^j &= \mathcal{H}_{\theta}^j \ominus z_j  \mathcal{R}_{\theta} \qquad j=1,2\\
%\mathcal{N}_{\theta}^j &= \mathcal{H}_{\theta}^j \ominus   \mathcal{R}_{\theta} \qquad j=1,2
%\end{aligned}
%\]
%Then we have the following decompositions
%\[
%\begin{aligned}
%\mathcal{K}_{\theta} &=  \left( \bigoplus_{n=0}^{\infty} z_1^n \mathcal{M}_{\theta}^1 \right) \oplus 
%\left( \bigoplus_{n=0}^{\infty} z_2^n \mathcal{N}_{\theta}^2 \right)  \\
%&=  \left( \bigoplus_{n=0}^{\infty} z_1^n \mathcal{N}_{\theta}^1 \right) \oplus 
%\left( \bigoplus_{n=0}^{\infty} z_2^n \mathcal{M}_{\theta}^2 \right).  \\
%\end{aligned}
%\]
Namely, define $\Smax_1$ to be the largest subspace of $\mathcal{K}_{\theta}$ invariant under multiplication by $z_1$, set $\Smin_2 \equiv \mathcal{K}_{\theta} \ominus \Smax_1$, and define $\Smax_2$, $\Smin_1$ analogously. Then, the pairs $(\Smax_1, \Smin_2)$ and $(\Smin_1, \Smax_2)$ are Agler subspaces of $\theta$ and if we set
\[
\begin{aligned}
\mathcal H(K_j^{max}) &= \Smax_j \ominus z_j \Smax_j ; \\
\mathcal H(K_j^{min}) &= \Smin_j \ominus z_j \Smin_j, 
\end{aligned}
\]
then $(K^{max}_1, K^{min}_2)$ and $(K^{min}_1, K^{max}_2)$ are pairs of Agler kernels of $\theta.$ This construction first appeared in \cite{bsv05} and was further studied in \cite{b12, bk13}. Our investigations are motivated by the enlightening situation when $\theta$ is a product of one variable inner functions. In this case, we can derive exact formulas for $K^{max}_j$ and $K^{min}_j$, which allow us to deduce numerous properties about $S_{z_1}$ and $S_{z_2}$. Much of the paper involves the best-known generalizations of these results.

%We also define  $K_j^{max}$ and $K_j^{min}$ to be the reproducing kernels of $ \mathcal{M}_{\theta}^j$ and $ \mathcal{N}_{\theta}^j$ respectively. 

%We can equivalently write these as tensor products as follows:
%\[
%\begin{aligned}
%\mathcal{H}_{\theta} &= \left( \mathcal{M}_{\theta}^1 \otimes H_1 ^2(\mathbb{D}) \right) \oplus 
%\left( \mathcal{N}_{\theta}^2 \otimes H_2 ^2(\mathbb{D}) \right) \\
% &= \left( \mathcal{N}_{\theta}^1 \otimes H_1 ^2(\mathbb{D}) \right) \oplus 
%\left( \mathcal{M}_{\theta}^2 \otimes H_2 ^2(\mathbb{D}) \right),
%\end{aligned}
%\]
%which may end up being a more useful representation when we want to discuss operators on these Hilbert spaces.\\
\subsection{Outline of Paper.}
In Section \ref{s-exa}, we restrict attention to $\theta$ that are products of one variable inner functions $\phi$ and $\psi$. 
While many computations and constructions can be done explicitly, this situation is quite non-trivial. We first obtain nice formulas for the Agler kernels $K^{max}_j$ and $K^{min}_j$, which allow us to get explicit formulas for the compressed shifts $S_{z_1}$ and $S_{z_2}$. 
A study of these formulas shows that the subspaces $(\Smax_1,\Smin_2)$ are reducing for $S_{z_1}$ and $(\Smin_1, \Smax_2)$ are reducing for $S_{z_2}.$ For details, see Proposition \ref{p-representation}. Further, interestingly, the essential normality of $S_{z_1}$ and $S_{z_2}$ has a simple characterization in terms of the structure of $\phi$ and $\psi$, see Proposition \ref{t-prodessnorm}. We are also able to study the spectrum of $S_{z_1}$ and $S_{z_2}$.  The results in this section provide motivation for the sections to come, where we obtain analogues of both the essential normality result and the reducing subspaces result for more general inner functions. 

First, in Section \ref{s-general}, we obtain the following generalization of Proposition \ref{t-prodessnorm}. The most surprising outcome is that our generalized arguments now characterize finite rank, rather than compactness, of the commutator. Specifically, we use Agler decompositions of $\theta$ to establish:
\begin{theorem}\label{t-IFF}
Let $\theta$ in $H^2(\mathbb{D}^2)$ be inner. Then the  commutator $[S_{z_1}^*,S_{z_1}]$ has rank $n$ if and only if $\theta$ is a rational inner function of degree $(1,n)$ or $(0,n)$.
\end{theorem}
Observe that this result complements the Guo and Wang result from \cite{gw09} discussed earlier.  In fact, Theorem \ref{t-IFF} together with Guo-Wang's result implies that if $S_{z_1}$ and $S_{z_2}$ are simultaneously essentially normal, then the two commutators are actually at most rank one!

Second, in Section \ref{s-reducing} we study when Agler subspaces are reducing for either of the compressed shifts. First, we determine conditions for the Agler subspaces 
\[ \mathcal H\left(\frac{K_1(z,w)}{1-z_1\bar{w}_1}\right) \text{ and }\, \mathcal H\left(\frac{K_2(z,w)}{1-z_2\bar{w}_2}\right)\]
 to be reducing for the compressed shift operators in terms of the kernels $K_1$ and $K_2$, see Theorem \ref{thm:Reducing1}. A subtle relationship (see Theorem \ref{thm:Reducing2}) between properties of the kernels and the properties of $\theta$ allows us to conclude that, if $\theta$ is rational inner, then the products in Section \ref{s-exa} are the only inner functions with Agler subspaces as reducing subspaces:
 
 \begin{theorem}
 \label{t-iff} Let $\theta$ be a rational inner function on $\mathbb{D}^2.$ Then $\theta$ has a pair of Agler kernels $(K_1, K_2)$ such that the associated Agler spaces
 \[ \mathcal{S}_1 \equiv \mathcal{H} \left(\frac{K_1(z,w)}{1-z_1\bar{w}_1} \right)  \text{ and }  \,\mathcal{S}_2 \equiv\mathcal{H} \left(\frac{K_2(z,w)}{1-z_2\bar{w}_2} \right)  \]
 are reducing subspaces for $S_{z_1}$ if and only if $\theta$ is a product of one variable inner functions.
 \end{theorem}

At the end of both Sections \ref{s-general} and \ref{s-reducing}, we include related open questions. The authors are currently investigating the situation where $\theta$ is matrix-valued. Results in this setting will appear in a later publication.

\textit{Acknowledgements.} The authors would like to thank R.G.~Douglas and R.~Yang for inspiring conversations. As this work was initiated at the Oberwolfach workshop ``Hilbert Modules and Complex Geometry,'' the authors would also like to thank the MFO (Oberwolfach) and the workshop organizers for providing a stimulating environment.

\section{A First Example}\label{s-exa}

In this section, we consider  $\theta(z) = \phi(z_1) \psi(z_2)$, for one variable inner functions $\phi$ and $\psi$, and  the compressed shift operators $S_{z_1}$ and $S_{z_2}$ on $\mathcal{K}_{\theta}.$ Even in this simple situation, there is much to be said. 

\subsection{Agler decompositions of $\theta$}
Before examining $S_{z_1}$ and $S_{z_2}$, we obtain nice formulas for the shift-invariant subspaces $\mathcal{S}^{max}_1$ and $\mathcal{S}^{min}_2.$ First, observe that by adding and subtracting $\psi(z_2) \overline{\psi(w_2)} $ in the numerator, one obtains:
\[ \frac{1- \theta(z)\overline{\theta(w)}}{(1-z_1\bar{w}_1)(1-z_2\bar{w}_2)} =
\frac{1 - \psi(z_2) \overline{ \psi(w_2)}}{(1-z_1\bar{w}_1)(1-z_2 \bar{w}_2)} 
+   \frac{ \psi(z_2) \overline{\phi(w_2)} (1 - \phi(z_1) \overline{ \phi(w_1)})}{(1-z_1\bar{w}_1)(1-z_2\bar{w}_2)}. \]
Both terms on the right-hand-side of the equation are positive kernels and indeed, it turns out that they are the reproducing kernels for $\mathcal{S}^{max}_1$ and $\mathcal{S}^{min}_2$ respectively. It suffices to prove the claim for $\mathcal{S}^{max}_1$, as $\mathcal{S}^{min}_2 = \mathcal{K}_{\theta} \ominus \mathcal{S}^{max}_1.$ Now, it is clear that:
\[  
\mathcal{H} \left( \frac{1 - \psi(z_2) \overline{ \psi(w_2)}}{(1-z_1\bar{w}_1)(1-z_2 \bar{w}_2)}\right) =H^2(\mathbb{D}^2) \ominus \psi(z_2) H^2(\mathbb{D}^2)
\]
is a subspace of $\mathcal{K}_{\theta}$ and by the $(1-z_1\bar{w}_1)$ in its kernel's denominator, is invariant under multiplication by $z_1.$ Hence, it is contained in $\mathcal{S}^{max}_1.$ To see equality, assume
\[ 
f \in \Smax_1  \ \text{ and } \  f \perp H^2(\mathbb{D}^2) \ominus \psi(z_2) H^2(\mathbb{D}^2).
\] 
Then $f= \psi(z_2)g$ for some $g \in H^2(\mathbb{D}^2)$. But, since $z_1^n f \in \Smax_1$ for all $n$, this would imply $\phi(z_1)\psi(z_2)g = \theta g \in \Smax_1 \subseteq \mathcal{K}_{\theta}.$ By the definition of $\mathcal{K}_{\theta}$, this implies $f \equiv 0$ and so, we have found $\Smax_1.$ Now, it is immediate that
\[ \Smin_2 = \mathcal{K}_{\theta} \ominus \Smax_1 = \mathcal{H}\left( \frac{  \psi(z_2) \overline{\psi(w_2)} \left(1 - \phi(z_1) \overline{ \phi(w_1)}\right)}{(1-z_1\bar{w}_1)(1-z_2 \bar{w}_2)}\right).\]
To simplify notation, define the one variable model spaces
\[ \mathcal{K}^2_{\psi} \equiv H_2^2(\mathbb{D}) \ominus \psi(z_2) H_2^2(\mathbb{D}) \ \text{ and } \ \mathcal{K}^1_{\phi} \equiv H_1^2(\mathbb{D}) \ominus \phi(z_1) H_1^2(\mathbb{D}),  \]
where $H^2_j(\mathbb{D})$ is the one variable Hardy space with independent variable $z_j$ for $j=1,2.$ Then,  if $(K^{max}_1, K^{min}_2)$ are the associated Agler kernels of $\theta$,
\[ 
\begin{aligned}
 \mathcal{H}(K^{max}_1) &= \Smax_1 \ominus z_1 \Smax_1 = H_2^2(\mathbb{D}) \ominus \psi(z_2) H_2^2(\mathbb{D})= \mathcal{K}^2_{\psi}; \\
 \mathcal{H}(K^{min}_2) &= \Smin_2 \ominus z_2 \Smin_2 = 
\psi(z_2) \left[ H^2_1(\mathbb{D}) \ominus \phi(z_1) H^2_1(\mathbb{D})\right] =  \psi(z_2) \mathcal{K}^1_{\phi}.
 \end{aligned}
 \]
Symmetric formulas for the subspaces $\mathcal{S}^{min}_1$, $\mathcal{S}^{max}_2$ and Agler kernels $(K^{min}_1, K^{max}_2)$ can be obtained in a similar fashion.  These decompositions of $\mathcal{K}_{\theta}$ into complementary spaces are summarized in the following proposition:
\begin{proposition}
If $\theta (z) = \phi(z_1) \psi(z_2)$, then
\[
\begin{aligned}
\mathcal{K}_{\theta} &= \mathcal{S}^{max}_1 \oplus \mathcal{S}^{min}_2 = \left[ \mathcal{K}^2_{\psi} \otimes   H^2_1(\mathbb{D})\right] \oplus  \left[ \mathcal{K}^1_{\phi} \otimes   \psi H^2_2(\mathbb{D})\right] \\
&= \mathcal{S}^{min}_1 \oplus \mathcal{S}^{max}_2 = \left[ \mathcal{K}^2_{\psi} \otimes  \phi H^2_1(\mathbb{D})\right] \oplus  \left[ \mathcal{K}^1_{\phi} \otimes   H^2_2(\mathbb{D})\right].
\end{aligned}
\]
\end{proposition}
This result proves useful when studying the operators of interest on $\mathcal{K}_{\theta}.$

\subsection{The compression of the shift operators on $\mathcal{K}_{\theta}$}\label{ss-prodreducing}
In the product setting the compressed shift operators $S_{z_1} = P_{\theta} T_{z_1}|_{\mathcal{K}_{\theta}}$ and $S_{z_2} = P_{\theta} T_{z_2}|_{\mathcal{K}_{\theta}}$ have a special form.

\begin{proposition}\label{p-representation}
If $\theta (z) = \phi(z_1) \psi(z_2)$, then
\[ S_{z_1} = \left[ \begin{array}{cc} T_{z_1}|_{\Smax_1} & 0\\
0 & P_{\mathcal{K}^1_{\phi}} T_{z_1} \otimes I_{\psi H^2_2(\mathbb{D})} \end{array} \right ]: 
\left[ \begin{array}{r}
\Smax_1 \\ \Smin_2
\end{array} \right]
\rightarrow 
\left[ \begin{array}{r}
\Smax_1 \\ \Smin_2 
\end{array} \right]
 \]
and
\[ S_{z_2} = \left[ \begin{array}{cc} P_{\mathcal{K}^2_{\psi}} T_{z_2} \otimes I_{\phi H^2_1(\mathbb{D})} & 0\\
0 & T_{z_2}|_{\Smax_2}   \end{array} \right ]: 
\left[ \begin{array}{r}
\Smin_1 \\ \Smax_2
\end{array} \right]
\rightarrow 
\left[ \begin{array}{r}
\Smin_1 \\ \Smax_2 
\end{array} \right].
 \]
In particular, the spaces $\Smax_1$, $\Smin_2$ are reducing for $S_{z_1}$ and the spaces $\mathcal{S}^{min}_1$, $\mathcal{S}^{max}_2$ are reducing for $S_{z_2}$. 
\end{proposition}

\begin{proof}
Consider $S_{z_1}$ on $\mathcal{K}_{\theta}.$ For ease of notation, we write $\M \equiv \Smax_1$ and $\N \equiv \Smin_2.$ Then we can write $S_{z_1}$ as a block operator:
\[ S_{z_1} = \left[ \begin{array}{rr} P_{\M} T_{z_1} P_{\M} & P_{\M} T_{z_1} P_{\N} \\
P_{\N} T_{z_1} P_{\M} & P_{\N} T_{z_1} P_{\N} \end{array} \right ]: 
\left[ \begin{array}{r}
\M \\ \N 
\end{array} \right]
\rightarrow 
\left[ \begin{array}{r}
\M \\ \N 
\end{array} \right].
 \]
Now, we simply study each of these operators separately. First since $\M$ is invariant under multiplication by $z_1,$ we have
\[ P_{\M} T_{z_1} P_{\M} = T_{z_1}|_\M \ \text{ and } \ P_{\N} T_{z_1} P_{\M} =0.\]
Furthermore, observe that, if $f \in \N$, then $f = \psi(z_2)g(z)$ for some $g \in H^2(\mathbb{D}^2) \ominus \phi(z_1) H^2(\mathbb{D}^2) .$ Then
\[ T_{z_1} f(z) = \psi(z_2) z_1 g(z) \ \text{ is in } \ \psi(z_2) H^2(\mathbb{D}^2).\]
 Since $\M = H^2(\mathbb{D}^2) \ominus \psi(z_2) H^2( \mathbb{D}^2)$, we can conclude
\[ T_{z_1} f \perp \M  \ \text{ which implies }  \ P_{\M} T_{z_1} P_{\N} =0. \]
Lastly, recall that $\N = \mathcal{K}^1_{\phi} \otimes \psi(z_2)H^2_2(\mathbb{D}).$ Then, fix $(f(z_1) \otimes \psi(z_2) g(z_2) )$ in $\N$ and observe that
\[ 
\begin{aligned}
T_{z_1}(f(z_1) \otimes \psi(z_2) g(z_2)) &= (T_{z_1}f(z_1) \otimes \psi(z_2) g(z_2)) \\
& = (P_{\mathcal{K}^1_{\phi}} T_{z_1}f(z_1) \otimes \psi(z_2) g(z_2))  + 
(P_{\phi H^2_1(\mathbb{D})} T_{z_1}f(z_1) \otimes \psi(z_2) g(z_2)).
\end{aligned}
\]
It is easy to show that the second piece of the sum is orthogonal to $\N$, so 
\[ P_{\N} T_{z_1}(f(z_1) \otimes \psi(z_2) g(z_2)) = (P_{\mathcal{K}^1_{\phi}} T_{z_1}f(z_1) \otimes \psi(z_2) g(z_2)).\] 
Since linear combinations of elements of the form $(f(z_1) \otimes \psi(z_2) g(z_2) )$ are dense in $\N$, we have
\[ P_{\N} T_{z_1} P_{\N} = P_{\mathcal{K}^1_{\phi}} T_{z_1} \otimes I_{\psi H^2_2(\mathbb{D})}.\]
The formula for $S_{z_2}$ holds by analogy and the third statement follows immediately from the expressions for $S_{z_1}$ and $S_{z_2}$.
\end{proof}

\subsection{Characterizing essential normality}\label{ss-exa-en}
In this particular situation, it is not hard to study the essential normality of $S_{z_1}$ and similarly, of $S_{z_2}.$  The result is the following:

\begin{proposition}\label{t-prodessnorm} If $\theta(z) = \phi(z_1) \psi(z_2)$, then
$S_{z_1}$ is essentially normal iff $\psi$ is a finite Blaschke product and $\phi$ is a single Blaschke factor. Similarly, $S_{z_2}$ is essentially normal iff $\phi$ is a finite Blaschke product and $\psi$ is a single Blaschke factor. 
\end{proposition}

\begin{proof} To show $S_{z_1}$ is essentially normal, we need to show $[S^*_{z_1}, S_{z_1}]$ is compact. As $\Smax_1, \Smin_2$ are reducing for $S_{z_1},$ the commutator $[S^*_{z_1}, S_{z_1}]$ is compact iff both 
\[ [S^*_{z_1}, S_{z_1}] |_{\Smax_1} \text{ and } [S^*_{z_1}, S_{z_1}] |_{\Smin_2} \]
are compact. Thus, we can study those restricted operators separately. 

\noindent \textbf{Component 1: $[S^*_{z_1}, S_{z_1}] |_{\Smax_1}.$} Using the formulas from Proposition \ref{p-representation}, we have:
\[ 
S_{z_1} |_{\Smax_1} = I_{\mathcal{K}^2_{\psi}} \otimes T_{z_1}|_{H^2_1(\mathbb{D})} \text{ and } 
\left( S_{z_1} |_{\Smax_1} \right)^*=I_{\mathcal{K}^2_{\psi}} \otimes T_{\bar{z}_1}|_{H^2_1(\mathbb{D})}.
\]
Then, since $\Smax_1$ and $\Smin_2$ are reducing for $S_{z_1},$ we can conclude:
\[
\begin{aligned}
\left [S^*_{z_1}, S_{z_1} \right ] |_{\Smax_1} &= \left( S_{z_1} |_{\Smax_1} \right)^* \left( S_{z_1}|_{\Smax_1} \right) - \left(S_{z_1}|_{\Smax_1} \right) \left( S_{z_1} |_{\Smax_1} \right)^* \\
 &= \left[  I_{\mathcal{K}^2_{\psi}} \otimes T_{\bar{z}_1}|_{H^2_1(\mathbb{D})}\right] \left[ I_{\mathcal{K}^2_{\psi}} \otimes T_{z_1}|_{H^2_1(\mathbb{D})} \right] - \left[ I_{\mathcal{K}^2_{\psi}} \otimes T_{z_1}|_{H^2_1(\mathbb{D})} \right] \left[ I_{\mathcal{K}^2_{\psi}} \otimes T_{\bar{z}_1}|_{H^2_1(\mathbb{D})} \right] \\
&=  I_{\mathcal{K}^2_{\psi}} \otimes  \left[ T_{\bar{z}_1} T_{z_1} - T_{z_1} T_{\bar{z}_1} \right].
\end{aligned}
\]
Now fix $f \in H^2_1(\mathbb{D})$. Then
\[ 
\left[ T_{\bar{z}_1} T_{z_1} - T_{z_1} T_{\bar{z}_1} \right] f = f - \left[ f - f(0) \right ] = f(0).\]
So, letting $M_{\lambda}$ denote point evaluation at $\lambda$ on $H^2_1(\mathbb{D})$, we have
\[
\left [S^*_{z_1}, S_{z_1} \right ] |_{\Smax_1}   = I_{\mathcal{K}^2_{\psi}} \otimes M_0. 
\]
Now, we show this operator is compact iff $\psi$ is a finite Blaschke product. First, if $\psi$ is not a finite Blaschke product, then $\mathcal{K}^2_{\psi}$ is infinite dimensional and one can choose an infinite orthonormal basis $\{ f_n\}$ of $\mathcal{K}^2_{\psi}$. Then the sequence $\{ f_n \otimes 1 \}$ is bounded in $\Smax_1$ but the sequence
\[
\{ \left [S^*_{z_1}, S_{z_1} \right ] |_{\Smax_1} (f_n \otimes 1) \} = \{ f_n \otimes 1 \}
\]
does not have a convergent subsequence. Similarly, if $\psi$ is a finite Blaschke product, then $\mathcal{K}^2_{\psi}$ is finite dimensional. This implies $\left [S^*_{z_1}, S_{z_1} \right ] |_{\Smax_1}$ is finite rank and hence, compact. \\

\noindent \textbf{Component 2:  $[S^*_{z_1}, S_{z_1}] |_{\Smin_2}.$ } Using the formulas from Proposition \ref{p-representation}, we have  
\[  
S_{z_1}|_{\Smin_2} = P_{\mathcal{K}^1_{\phi}} T_{z_1} \otimes I_{\psi H^2_2(\mathbb{D})} \text{ and }  \left( S_{z_1}|_{\Smin_2} \right)^* = T_{\bar{z}_1}|_{\mathcal{K}^1_{\phi}}  \otimes I_{\psi H^2_2(\mathbb{D})},  \]
where the last formula uses the fact that $\mathcal{K}^1_{\phi}$ is invariant under $T_{\bar{z}_1}$.
Again, since $\Smin_2$ is reducing for $S_{z_1}$, we can write:
\[ 
\begin{aligned}
\left [S^*_{z_1}, S_{z_1} \right] |_{\Smin_2}
 &=  \left( S_{z_1} |_{\Smin_2} \right)^* \left( S_{z_1}|_{\Smin_2} \right) - \left(S_{z_1}|_{\Smin_2} \right) \left( S_{z_1} |_{\Smin_2} \right)^* \\
 & =\left( T_{\bar{z}_1}|_{\mathcal{K}^1_{\phi}}  \otimes I_{\psi H^2_2(\mathbb{D})} \right) \left(P_{\mathcal{K}^1_{\phi}} T_{z_1} \otimes I_{\psi H^2_2(\mathbb{D})}\right) -\left(T_{\bar{z}_1}|_{\mathcal{K}^1_{\phi}}  \otimes I_{\psi H^2_2(\mathbb{D})} \right) \left(
 P_{\mathcal{K}^1_{\phi}} T_{z_1} \otimes I_{\psi H^2_2(\mathbb{D})})\right)  \\
 & = \left[ \left(T_{\bar{z}_1}|_{\mathcal{K}^1_{\phi}} \right) \left(P_{\mathcal{K}^1_{\phi}} T_{z_1} \right) - \left(P_{\mathcal{K}^1_{\phi}} T_{z_1}  \right) \left(T_{\bar{z}_1}|_{\mathcal{K}^1_{\phi}}  \right)\right ] \otimes I_{\psi H^2_2(\mathbb{D})}. 
 \end{aligned}
 \]
Moreover, by Theorem  (II-9) in \cite{sar94}, for each $f \in \mathcal{K}^1_{\phi}$, 
\[  \left( T_{\bar{z}_1}|_{\mathcal{K}^1_{\phi}} \right)^* f = M_{z_1} f - \left \langle f, T_{\bar{z}_1} \phi \right \rangle_{H^2_1(\mathbb{D})} \phi,\]
which, by uniqueness,  implies that
\[  P_{\mathcal{K}^1_{\phi}} T_{z_1} f= M_{z_1} f - \left \langle f, T_{\bar{z}_1} \phi \right \rangle_{H^2_1(\mathbb{D})} \phi. \]
Now, we show that $\left [S^*_{z_1}, S_{z_1} \right] |_{\Smin_2}$ is compact iff
\[ C \equiv \left(T_{\bar{z}_1}|_{\mathcal{K}^1_{\phi}} \right) \left( P_{\mathcal{K}^1_{\phi}} T_{z_1} \right) - \left(P_{\mathcal{K}^1_{\phi}} T_{z_1} \right) \left( T_{\bar{z}_1}|_{\mathcal{K}^1_{\phi}} \right) \equiv 0. \]
If $C \equiv 0,$ then $\left [S^*_{z_1}, S_{z_1} \right] |_{\Smin_2} \equiv 0$, and so is clearly compact. Now assume $C$ is nonzero. Specifically, assume there is some function $g$,  such that $Cg = h$ for some nonzero function $h$. Choose a sequence of orthonormal vectors $\{f_n\}$ in $\psi H^2_2(\mathbb{D}).$ Then the sequence $\{ g \otimes f_n \}$ is bounded in $\Smin_2$ but 
\[ \left \{ \left [S^*_{z_1}, S_{z_1} \right] |_{\Smin_2}( g \otimes f_n) \right\} = \{ h \otimes f_n \} \]
does not have a convergent subsequence, so the operator is not compact. To finish the characterization, fix $f \in \mathcal{K}^1_{\phi}.$ Then
\[
\left( T_{\bar{z}_1}|_{\mathcal{K}^1_{\phi}} \right) \left( P_{\mathcal{K}^1_{\phi}} T_{z_1} \right) f = T_{\bar{z}_1} \left( z_1 f  - \left \langle f, T_{\bar{z}_1} \phi \right \rangle_{H^2_1(\mathbb{D})} \phi \right) \\
 = f -  \left \langle f, T_{\bar{z}_1} \phi \right \rangle_{H^2_1(\mathbb{D})} T_{\bar{z}_1} \phi.
\]
Similarly, we have
\[ 
 \left( P_{\mathcal{K}^1_{\phi}}  T_{z_1} \right) \left( T_{\bar{z}_1}|_{\mathcal{K}^1_{\phi}} \right) f = f - f(0) - \left \langle T_{\bar{z}_1} f, T_{\bar{z}_1}\phi \right \rangle_{H^2_1(\mathbb{D})} \phi. \] 
So, it follows that
\[ Cf = f(0) - \left \langle T_{\bar{z}_1} f, T_{\bar{z}_1}\phi \right \rangle_{H^2_1(\mathbb{D})} \phi + \left \langle f, T_{\bar{z}_1} \phi \right \rangle_{H^2_1(\mathbb{D})} T_{\bar{z}_1} \phi. \]
Now we show that $C \equiv 0$ iff $\phi$ is a single Blaschke factor.   Now, if $\phi$ is a single Blaschke factor, then  $\mathcal{K}^1_{\phi}$ is a one-dimensional space. This implies that all linear operators on $\mathcal{K}^1_{\phi}$ are trivially normal, so $C \equiv 0.$ Now assume $C \equiv 0,$ and trivially, $\phi$ is not a constant function. Since $f =T_{\bar{z}_1} \phi \in \mathcal{K}^1_{\phi}$ is nonzero, $C \equiv 0$ implies there are constants $a, b,c$ with $c \ne 0$ such that
\[ 0= a - b \phi + c T_{\bar{z}_1} \phi\]
which implies
\[T_{\bar{z}_1} \phi (z_1) = \frac{1}{c} \left( b\phi(z_1) -a \right).\]
By definition, we also know:
\begin{equation} \label{eqn:bws} T_{\bar{z}_1}\phi(z_1)  = \frac{\phi(z_1) - \phi(0)}{z_1}. \end{equation}
Setting these two equations equal and solving for $\phi$ gives:
\[  \phi(z_1) = \frac{c \phi(0)- az_1}{c - bz_1}. \]
Since $\phi$ is an inner function, this implies $\phi$ is a single Blaschke factor. Thus, $[S^*_{z_1},S_{z_1}]|_{\Smin_2}$ is compact iff $C \equiv 0$, which is true iff $\phi$ is a single Blaschke factor.

Combining the conditions for  $[S^*_{z_1},S_{z_1}]|_{\Smax_1} $ and $[S^*_{z_1},S_{z_1}]|_{\Smin_2} $ to be compact, we get that $[S^*_{z_1},S_{z_1}]$ is compact iff $\psi$ is a finite Blaschke product and $\phi$ is a single Blaschke factor. \end{proof}

The following considerations yield an easy corollary to the previous proof. Specifically, the proof showed that if $S_{z_1}$ is essentially normal, then
\[[S^*_{z_1}, S_{z_1}] = \left( I_{\mathcal{K}^2_{\psi}} \otimes M_0 \right) |_{\Smax_1} \oplus 0|_{\Smin_2}. \]
where $\psi$ is a finite Blaschke product and so, $\mathcal{K}^2_\psi$ is a finite dimensional vector space. This immediately gives:

\begin{corollary}\label{c-essnorm}
Assume $\theta(z) = \phi(z_1)\psi(z_2)$, for $\phi$ and $\psi$ one variable and inner.  Then the essential normality of $S_{z_1}$ on $\mathcal{K}_{\theta}$ implies that $\textrm{rank}\,[S_{z_1}^*,S_{z_1}]<\infty$.
\end{corollary}

\subsection{Operator-theoretic and spectral properties}\label{ss-speccomm}
In this subsection, we make the common assumption that $\phi$ and $\psi $ are contractions, i.e., $|\phi(0)|<1$ and $|\psi(0)|<1$. If both functions are pure, namely $| \phi(0)| = |\psi(0)|=1$, then $\theta$ is a rotation and $\mathcal{K}_{\theta}$ is trivial. If only one of the functions is not pure, then the problem still simplifies, just not as drastically. This situation is addressed in the remark following the proof of Proposition \ref{p-spectrum}.

In the situation, using the representation formulas of the compressed shift operators $S_{z_1}$ and $S_{z_2}$ given in Proposition \ref{p-representation}, we obtain the following operator-theoretic and spectral properties:

\noindent\begin{proposition}\label{p-spectrum}
Assume $\theta (z) = \phi(z_1) \psi(z_2)$ and that $\phi$ and $\psi$ are one variable inner contractions. Then:
\begin{itemize}
\item[(a)] The first component of $S_{z_1}$ is an isometry and the second component is cnu (i.e., a completely non-unitary contraction).
\item[(b)] The spectrum of $S_{z_1}$
$
\sigma\left(S_{z_1} \right) = \overline{\mathbb{D}^2}. 
$
More precisely, 
\[
\sigma\left(S_{z_1}|_{\Smax_1}\right)=\sigma_{ac}\left(S_{z_1}|_{\Smax_1}\right)=\overline{\mathbb{D}^2}
\quad\text{and}
\quad
\sigma\left(S_{z_1}|_{\Smin_2}\right)=\sigma_{ess}\left(S_{z_1}|_{\Smin_2}\right)=\sigma(\phi).
\]
\item[(c)] The commutator $[S_{z_1}^*,S_{z_1}]|_{\Smax_1}$ has eigenvalue 1 with multiplicity equal to the dimension of $\mathcal{K}_\psi^2$. In particular, the multiplicity is finite if and only if $\psi$ is a finite Blaschke product. Further, $[S_{z_1}^*,S_{z_1}]|_{\Smin_2}$ has three eigenvalues of infinite multiplicity.\end{itemize}
Operator $S_{z_2}$ possesses the analogous properties.
\end{proposition}

\begin{proof}
(a) Recall the representation of $S_{z_1}$ from Proposition \ref{p-representation}. The first component $T_{z_1}:\mathcal{S}_1^{max}\to \mathcal{S}_1^{max}$ is an isometry, since $\mathcal{S}_1^{max}$ is invariant under multiplication by $z_1$ and not surjective.
In the second component, the first factor $P_{\mathcal{K}^1_{\phi}} T_{z_1}$ is exactly the compression of the shift operator on the model space $\mathcal{K}^1_{\phi}.$ So model theory \cite{SzNF2010} informs us that we have a cnu contraction. 
The second factor is the identity on $\psi H^2_2(\mathbb{D})$ and does not influence this component. Hence the second component is a cnu contraction.

(b) The representation of $S_{z_1}$ tells us that $S_{z_1}|_{\Smax_1}$ is the shift operator $T_{z_1}$. So, the spectrum of this part equals $\overline{\mathbb{D}^2}$ and is purely absolutely continuous. For the second component, recall that the spectrum of the model operator $P_{\mathcal{K}^1_\phi} T_{z_1}$ on $\mathcal{K}^1_\phi$ equals that of the inner function $\phi$. More precisely (see e.g.~\cite{rg12}), the point spectrum of $P_{\mathcal{K}^1_\phi} T_{z_1}$ equals $\sigma(\phi)\cap \mathbb{D}$ and the essential spectrum is equal to $\sigma(\phi)\cap\partial \mathbb{D}$. And, since $\psi H^2_2(\mathbb{D})$ is infinite dimensional, even the discrete spectrum of $P_{\mathcal{K}^1_\phi} T_{z_1}$ yields essential spectrum for $P_{\mathcal{K}^1_\phi} T_{z_1} \otimes I_{\psi H^2_2(\mathbb{D})}$. To see this, let $f$ be an eigenfunction of $P_{\mathcal{K}^1_\phi} T_{z_1}$ and $\{g_n\}$ be a basis of $\psi H^2_2(\mathbb{D})$. Then $\{f\otimes g_n\}$ is an infinite linearly independent sequence of eigenfunctions for $P_{\mathcal{K}^1_\phi} T_{z_1} \otimes I_{\psi H_2^2(\mathbb{D})}$.
Therefore, the spectrum of $P_{\mathcal{K}^1_\phi} T_{z_1} \otimes I_{\psi H^2_2(\mathbb{D})}$ is essential and equals that of $\phi$.

(c) In the proof of Proposition \ref{t-prodessnorm} we computed
\begin{align*}
[S_{z_1}^*,S_{z_1}] 
&= 
[S_{z_1}^*,S_{z_1}]|_{\Smax_1}
\oplus
[S_{z_1}^*,S_{z_1}]|_{\Smin_2}
\\
&=
%[A,A^*]
%\oplus 
%[B,B^*]
%=
[I_{\mathcal{K}^2_\psi}\otimes M_0]
\oplus
[(M_0 - \langle T_{\bar{z}_1}\cdot, T_{\bar{z}_1}\phi\rangle\phi + \langle \cdot, T_{\bar z_1}\phi\rangle T_{\bar z_1}\phi)\otimes I_{\psi H_2^2(\mathbb D)}].
\end{align*}
Since $M_0$ is point evaluation at $0$, it is the rank one operator with the constant function as eigenvector and corresponding eigenvalue 1. Hence $[S_{z_1}^*,S_{z_1}]|_{\Smax_1}$ has eigenvalue 1 with multiplicity equal to the dimension of $\mathcal{K}_\psi^2$. On $\Smin_2$ the first factor is a rank three operator. Due to the second factor, each eigenvalue occurs with infinite multiplicity.
\end{proof}

\begin{remark} \textnormal{
Now, we briefly consider the situation where $\theta(z) = \phi(z_1)\psi(z_2)$ where $|\phi(z_1)|=1$ and $|\psi(z_2)|<1.$ As $|\phi(z_1)|=1$, it follows that $\mathcal{K}_{\phi}^1 =\{0\}.$ This means that the second components of both $S_{z_1}$ and $S_{z_2}$ are trivial. However, the first component of $S_{z_1}$ is still an isometry and the first component of $S_{z_2}$ is still cnu. Similarly, it is easy to see that in this case,
\[ \sigma(S_{z_1}) = \sigma( S_{z_1}|_{\Smax_1}) = \overline{\mathbb{D}^2} \ \text{ and }
\sigma(S_{z_2} )= \sigma(S_{z_2}|_{\Smin_1}) = \sigma(\psi). \]
Finally, in (c),  the commutator $[S_{z_1}^*,S_{z_1}]|_{\Smax_1}$ still has eigenvalue 1 with multiplicity equal to the dimension of $\mathcal{K}_\psi^2$. However, $[S_{z_1}^*,S_{z_1}]|_{\Smin_2}$ is trivial. Similarly,  the commutator $[S_{z_2}^*,S_{z_2}]|_{\Smax_2}$ is trivial, but 
 $[S_{z_2}^*,S_{z_2}]|_{\Smin_1}$ still has three eigenvalues with infinite multiplicity.}
\end{remark}

\section{The Degree of General Inner Functions and the Rank of Commutators --- The Proof of Theorem \ref{t-IFF}}\label{s-general}

We now consider the more general situation of an arbitrary inner function $\theta$ on the bidisk and the behavior of the compressed shifts $S_{z_j}$, $j=1,2$, on the model space $\mathcal{K}_{\theta}.$ Our goal in this section is a generalization of Proposition \ref{t-prodessnorm}, previously called Theorem \ref{t-IFF}.

We will first outline several auxiliary results that clarify the structure of $S_{z_1}$ and connect the structure of rational inner functions to properties of their Agler decompositions. We then prove Theorem \ref{t-IFF} in two steps and conclude with several open questions.

\subsection{Auxiliary results}
We require the following lemma, which is likely well-known and is contained for example, in \cite{b12}.

\begin{lemma} \label{lem:ShiftForm} Let $\theta$ be a two variable inner function on $\mathbb{D}^2$ and recall that $S_{z_1}^* = T_{\bar{z}_1}|_{\mathcal{K}_{\theta}}.$ Then for each $f \in \mathcal{K}_{\theta}$,
\[ 
S_{z_1}f(z)  = \left( S_{z_1}^* \right )^* f(z)  = z_1 f(z) - \left \langle f(w), \frac{T_{\bar{z}_1} \theta(w)}{1-\bar{z}_2 w_2} \right \rangle_{H^2} \theta(z).\]
Here, one should notice that the integration in the inner product is occurring with respect to the variable $w$.
\end{lemma}

The proof is a simple calculation, which we include for the reader's convenience. 

\begin{proof}
First, observe that when we
apply the backward shift $T_{\bar{z}_1}$ to the reproducing kernel of $\mathcal{K}_{\theta}$, we get:
\[
T_{\bar{z}_1} \frac{1-\overline{\theta(w)} \theta(z)}{(1-z_1 \bar{w}_1)(1-z_2 \bar{w}_2)} = \bar{w}_1\frac{1-\overline{\theta(w)} \theta(z)}{(1-z_1 \bar{w}_1)(1-z_2 \bar{w}_2)} -
\overline{\theta(w)} \frac{(T_{\bar{z}_1}\theta)(z) }{1-z_2\bar{w}_2}.\]
 Now, we can
calculate the adjoint of $S_{z_1}^*$. Let
$f \in \mathcal{K}_{\theta}$ and $w \in \mathbb{D}^2$. Then
\[
\begin{aligned}
 (S_{z_1}^*)^*f(z) &= \left \langle  (S_{z_1}^*)^*f(w),
 \frac{1-\overline{\theta(z)} \theta(w)}{(1-w_1 \bar{z}_1)(1-w_2 \bar{z}_2)} \right \rangle_{\mathcal{K}_{\theta}} \\ 
&= \left \langle f(w),  T_{\bar{z}_1} \frac{1-\overline{\theta(z)} \theta(w)}{(1-w_1 \bar{z}_1)(1-w_2 \bar{z}_2)} \right \rangle_{\mathcal{K}_{\theta}}  \\ 
&= \left \langle f(w) , \bar{z}_1\frac{1-\overline{\theta(z)} \theta(w)}{(1-w_1 \bar{z}_1)(1-w_2 \bar{z}_2)} -
\overline{\theta(z)} \frac{(T_{\bar{z}_1}\theta)(w) }{1-w_2\bar{z}_2} \right \rangle_{\mathcal{K}_{\theta}}  \\ 
& = z_1 f(z) -  \left \langle f(w), \frac{T_{\bar{z}_1} \theta(w)}{1-\bar{z}_2 w_2} \right \rangle_{H^2} \theta(z),
\end{aligned} 
\]
which is the desired formula.
\end{proof}

The proof of Theorem \ref{t-IFF} uses connections between rational inner functions, Agler kernels, and the structure of $S_{z_1}$. So, we need several results concerning properties of rational inner functions and their associated Agler kernels. 

First, given a polynomial $p$ with $\deg p = (m,n),$ define its reflection $\tilde{p}$ to be the polynomial $\tilde{p}(z) = z_1^m z_2^n \overline{p( \frac{1}{\bar{z}_1}, \frac{1}{\bar{z}_2})}.$ Then, a result due to Rudin \cite{Rud69} states that  if $\theta$ is rational inner, then there is an (almost) unique polynomial $p$ with no zeros on $\mathbb{D}^2$ such that
\[ \theta(z) = \frac{ \tilde{p}(z)}{p(z)}, \]
and $p$ and $\tilde{p}$ share no common factors. Given that, we can state the following result, which is encoded in Theorems 1.7 and 1.8 in \cite{bk13} as well as in a slightly different form in Proposition 2.5 in \cite{w10}.

\begin{theorem} \label{thm:rational} Let $\theta = \frac{ \tilde{p}}{p}$ be a rational inner function of degree $(m,n)$. Then 
\[ 
\begin{aligned}
\dim \mathcal{H}(K^{max}_1) &= \dim \mathcal{H}(K^{min}_1) =n \\
\dim \mathcal{H}(K^{max}_2) &= \dim \mathcal{H}(K^{min}_2) =m.
\end{aligned}
\]
Furthermore, if $f$ is a function in $\mathcal{H}(K^{max}_1) $ or $\mathcal{H}(K^{min}_1)$ then $f = \frac{q}{p}$ where $\deg q \le (m,n-1)$ and  if $g$ is a function in $\mathcal{H}(K^{max}_2) $ or $\mathcal{H}(K^{min}_2)$ then $g = \frac{r}{p}$, where $\deg r \le (m-1,n).$
\end{theorem}

The converse of this theorem is also true and follows from the representation of $\theta$ as a transfer function of a coisometry defined using its Agler kernels. Although likely well-known, a reading of the construction in Remark 5.2 in \cite{bk13} paired with the definition of a transfer function realization of $\theta$ will immediately reveal the following:

\begin{theorem} \label{thm:rational3} If $\theta$ is an inner function such that 
\[ 
\dim \mathcal{H}(K^{max}_1)=n < \infty \text{ and }  \dim \mathcal{H}(K^{min}_2) =m<\infty,
\]
then $\theta$ is a rational inner function. Theorem \ref{thm:rational} then immediately implies that
$\deg \theta = (m,n).$
\end{theorem}

In Section 2, when $\theta$ was a product of one variable inner functions, many arguments relied on the fact that $\Smax_j$ and $\Smin_j$ were closely related to one variable model spaces. We require the following generalization of those relationships for arbitrary inner functions, which appears as Theorem 1.6 in \cite{bk13}. 
 
 \begin{theorem}\label{thm:restriction} Let $\theta$ be an inner function on $\mathbb{D}^2.$ Then for  almost every $t \in \mathbb{T}$, the map
 \[ f \mapsto f(\cdot, t)\]
 embeds $\mathcal{H}(K^{max}_2)$ and $\mathcal{H}(K^{min}_2)$ isometrically into $H^2(\mathbb{T}) \ominus \theta( \cdot, t) H^2(\mathbb{T}).$ An analogous fact holds for $\mathcal{H}(K^{max}_1)$ and $\mathcal{H}(K^{min}_1).$
 \end{theorem}
 
The above theorem uses the fact that, for almost every $t$ in $\mathbb{T}$ there is an inner function, traditionally denoted $\theta(\cdot ,t)$, which has boundary values $\theta(t_1,t)$ for almost every $t_1$ in $\mathbb{T}$. 

\subsection{The proof of Theorem \ref{t-IFF}}\label{proof-IFF}
We begin with the $(\Leftarrow)$-statement, which is encoded in the following auxiliary theorem:
\begin{theorem}
\label{t-2n} Let $\theta$ be a rational inner function with $\deg \theta \le (1,n)$. Then the commutator $[S_{z_1}^*, S_{z_1}]$ on $\mathcal{K}_{\theta}$ has rank at most $n$. In particular, it is essentially normal.
\end{theorem}
\begin{proof} 
Let $\theta$ be a rational inner function of degree at most $(1,n)$. By Theorem \ref{thm:rational}, this means that there are functions $f_1, \dots, f_n$ with $\deg f \le (1,n-1)$ and $g$ with $\deg g \le (0,n)$ such that 
\[ 
K^{max}_1(z,w) = \sum_{i=1}^n \frac{f_i(z) \overline{f_i(w)}}{p(z)\overline{p(w)}}  \ \text{ and } \  K^{min}_2(z,w) =   \frac{ g(z) \overline{g(w)}}{p(z)\overline{p(w)}}.\]
We can choose these functions so that they are orthogonal and either normalized or trivial. This also gives:
\[ \Smax_1 = \mathcal{H} \left( \frac{ \sum_{i=1}^n f_i(z) \overline{f_i(w)}}{p(z)\overline{p(w)}(1-z_1 \bar{w}_1)} \right) 
 \ \text{ and } \ \Smin_2 = \mathcal{H} \left( \frac{ g(z) \overline{g(w)}}{p(z)\overline{p(w)}(1-z_2 \bar{w}_2)} \right).
  \]
 Then, for each $w$ in $\mathbb{D}^2,$ we can write the reproducing kernel of $\mathcal{K}_{\theta}$ as
\[
K_w(z) \equiv 
\frac{1 -\theta(z)\overline{\theta(w)}}{(1-z_1\bar{w}_1)(1-z_2\bar{w}_2)} = \frac{ \sum_{i=1}^n f_i(z) \overline{f_i(w)}}{p(z)\overline{p(w)}(1-z_1 \bar{w}_1)}+ \frac{ g(z) \overline{g(w)}}{p(z)\overline{p(w)}(1-z_2 \bar{w}_2)}.
\]
For ease of notation, set
\[ K^1_w(z) \equiv  \frac{ \sum_{i=1}^n f_i(z) \overline{f_i(w)}}{p(z)\overline{p(w)}(1-z_1 \bar{w}_1)} \text{ and } K^2_w(z) \equiv  \frac{ g(z) \overline{g(w)}}{p(z)\overline{p(w)}(1-z_2 \bar{w}_2)}\]
to be the reproducing kernels for $\mathcal{S}^{max}_1$ and $\mathcal{S}^{min}_2.$ We first obtain  formulas for 
 \[ \left[S^*_{z_1}, S_{z_1}\right] K^1_w \ \text{ and } \ \left[S^*_{z_1}, S_{z_1}\right] K^2_w. \]
Let us consider $K^1_w.$ Since $\Smax_1$ is invariant under multiplication by $z_1$,   
\[ S^*_{z_1} S_{z_1} K^1_w = K^1_w\] 
 and similarly, 
 \[  
 \begin{aligned}
 S_{z_1}S^*_{z_1}K^1_w &= P_{\theta} \left( T_{z_1} S^*_{z_1}K^1_w \right) = K^1_w - P_{\theta} \left( K^1_w(0, z_2) \right).
 \end{aligned}
 \]
Using the formula for $K^1_w$, it is clear that
\begin{equation} \label{eqn:Max}
  \left[S^*_{z_1}, S_{z_1}\right] K^1_w   = P_{\theta} \left( K^1_w(0, z_2) \right) =    P_{\theta} \left(  \sum_{i=1}^{n}\frac{f_i(0,z_2)}{p(0,z_2)} \frac{ \overline{f_i(w)}}{\overline{p(w)}} \right). \end{equation}
Now, consider $K_w^2$. First, for $K_1 \equiv K^{max}_1$ and $K_2 \equiv K^{min}_2$, we can 
  rewrite equation \eqref{eqn:akernels} as 
 \begin{eqnarray}
  \frac{z_1 \bar{w}_1 K^{min}_2(z,w)}{1-z_2\bar{w}_2} &=& K^{max}_1(z,w) + \frac{K^{min}_2(z,w)}{1-z_2 \bar{w}_2} + \frac{ \theta(z) \overline{\theta(w)}}{1-z_2 \bar{w}_2}  - \frac{1}{1-z_2\bar{w}_2} \nonumber \\
& =& \sum_{i=1}^n \frac{f_i(z)}{p(z)} \frac{\overline{ f_i(w)}}{\overline{p(w)}} + \frac{ g(z) \overline{g(w)}}{p(z) \overline{p(w)}(1-z_2 \bar{w}_2)} + \frac{ \theta(z) \overline{\theta(w)}}{1-z_2 \bar{w}_2}  - \frac{1}{1-z_2\bar{w}_2}  \nonumber \\
 & = &\sum_{i=1}^n \frac{f_i(z)}{p(z)} \bar{w}_1 \overline{  T_{\bar{z}_1}\tfrac{f_i}{p}(w)} + \frac{ \frac{g(z)}{p(z)}  \bar{w}_1\overline{  T_{\bar{z}_1}  \frac{g}{p}(w)}}{1-z_2 \bar{w}_2} + \frac{ \theta(z)\bar{w}_1 \overline{ T_{\bar{z}_1}\theta (w)}}{1-z_2 \bar{w}_2}. \label{eqn:Min1}  
 \end{eqnarray}
 The last equation follows from setting $w_1=0$ and observing that
\[ \frac{1}{1-z_2 \bar{w}_2}  = \sum_{i=1}^n \frac{f_1(z)}{p(z)} \frac{\overline{ f_i(0,w_2)}}{p(0,z_2)} + \frac{ g(z) \overline{g(0,w_2)}}{p(z) \overline{p(0,w_2)}(1-z_2 \bar{w}_2)} + \frac{ \theta(z) \overline{\theta(0,w_2)}}{1-z_2 \bar{w}_2}.\]
Since \eqref{eqn:Min1} is conjugate-analytic in $w$, we can divide both sides by $\bar{w}_1$ to obtain 
\begin{equation} \label{eqn:Min2}  z_1 \frac{K^{min}_2(z,w)}{1-z_2\bar{w}_2} =   \sum_{i=1}^n \frac{f_i(z)}{p(z)} \overline{ T_{\bar{z}_1} \tfrac{f_i}{p}(w)} + \frac{ \frac{g(z)}{p(z)} \overline{ T_{\bar{z}_1}  \frac{g}{p}(w)}}{1-z_2 \bar{w}_2} + \frac{ \theta(z) \overline{T_{\bar{z}_1} \theta (w)}}{1-z_2 \bar{w}_2}. \end{equation}
 Fixing $w$ and projecting onto $\mathcal{K}_{\theta}$ gives
 \[ S_{z_1} \frac{K^{min}_2(z,w)}{1-z_2\bar{w}_2} =\sum_{i=1}^n \frac{f_i(z)}{p(z)} \overline{ T_{\bar{z}_1} \tfrac{f_i}{p}(w)} + \frac{ \frac{g(z)}{p(z)} \overline{ T_{\bar{z}_1} \frac{g}{p}(w)}}{1-z_2 \bar{w}_2}. 
 \]
Recalling the definition of $K_w^2$ and applying $S_{z_1}^*$ gives:
 \[ S_{z_1}^* S_{z_1} K^2_w(z) =\sum_{i=1}^n T_{\bar{z}_1} \tfrac{f_i}{p}(z) \overline{ T_{\bar{z}_1} \tfrac{f_i}{p}(w)} + \frac{ T_{\bar{z}_1}\frac{g}{p}(z) \overline{ T_{\bar{z}_1} \frac{g}{p}(w)}}{1-z_2 \bar{w}_2}.  \]
 Now, we  consider $S_{z_1} S_{z_1}^* K^2_w.$  First, by Theorem \ref{thm:rational}, $\deg g \le (0, n)$ and so $g(z)$ is a function of only $z_2$. This means we can calculate:  
\[  S_{z_1}^* K_w^2(z) = \frac{ T_{\bar{z}_1}\frac{g}{p}(z) \overline{ \frac{g}{p}(w)}}{1-z_2 \bar{w}_2} =
 \frac{g(z)\overline{g(w)}}{\overline{p(w)}(1-z_2\bar{w}_2)} \left( \frac{ \tfrac{1}{p(z)} -  \tfrac{1}{p(0,z_2)}}{z_1} \right) = \frac{-T_{\bar{z}_1}p(z) }{p(0,z_2)} K_w^2(z). \]
Then, using \eqref{eqn:Min2} we have
\[ 
\begin{aligned}
z_1 S^*_{z_1} K_w^2(z) 
&=  \left( \frac{-T_{\bar{z}_1}p(z) }{p(0,z_2)} \right) z_1 \frac{K^{min}_2(z,w)}{1-z_2\bar{w}_2} \\
&=  \frac{-T_{\bar{z}_1}p(z) }{p(0,z_2)} \left( \sum_{i=1}^n\frac{f_i(z)}{p(z)} \overline{ T_{\bar{z}_1} \tfrac{f_i}{p}(w)} + \frac{ \frac{g(z)}{p(z)} \overline{ T_{\bar{z}_1} \frac{g}{p}(w)}}{1-z_2 \bar{w}_2}+ \frac{ \theta(z) \overline{T_{\bar{z}_1}\theta(w)}}{1-z_2 \bar{w}_2} \right).  
\end{aligned}
\]
Now, by Lemma 9.1 in \cite{bk13}, since $\theta = \tilde{p} /p$, we can choose $p$ to have finitely many zeros on $\mathbb{T}^2.$ Then by Proposition 4.9.1 in \cite{Rud69}, if $p(0,z_2)$ has a zero at $(0,\tau_2)$ for some $\tau_2$(which means $\theta$ has a singular point there), then $\theta$ has a singular point at $(\tau_1, \tau_2)$ for every $\tau_1$ in $\mathbb{T}.$ This cannot happen, as  the singular points of $\theta$ occur only at the finite number of zeros of $p$. Thus, $p(0,z_2)$ has no zeros on $\overline{\mathbb{D}}$ and so $\frac{1}{p(0,z_2)}$ is in $H^{\infty}(\mathbb{D}).$ Now, we can calculate: 
 \[ 
 \begin{aligned}
 S_{z_1} S^*_{z_1} K^2_w(z) &= P_{\theta} \left(  \frac{-T_{\bar{z}_1}p(z) }{p(0,z_2)}  \sum_{i=1}^n \frac{f_i(z)}{p(z)} \overline{ T_{\bar{z}_1} \tfrac{f_i}{p}(w)} + \frac{-T_{\bar{z}_1}p(z) }{p(0,z_2)}\frac{ \frac{g(z)}{p(z)} \overline{ T_{\bar{z}_1}  \frac{g}{p}(w)}}{1-z_2 \bar{w}_2} \right) \\
 & =  \sum_{i=1}^n P_{\theta} \left(\frac{-T_{\bar{z}_1}p(z) }{p(0,z_2)} \frac{f_i(z)}{p(z)} \right)  \overline{ T_{\bar{z}_1} \tfrac{f_i}{p}(w)}   +
 \frac{  \overline{T_{\bar{z}_1} \frac{g}{p}(z) T_{\bar{z}_1} \frac{g}{p}(w)}}{1-z_2 \bar{w}_2},
 \end{aligned}
 \]
since $-\frac{T_{\bar{z}_1}p(z)}{p(0,z_2)} \frac{g(z)}{p(z)} = T_{\bar{z}_1} \frac{g}{p}(z).$ So, we can immediately calculate
 \[
 \begin{aligned}
   \left[ S^*_{z_1}, S_{z_1} \right] K^2_w(z) &=  \left(\sum_{i=1}^n T_{\bar{z}_1} \tfrac{f_i}{p}(z) \overline{ T_{\bar{z}_1} \tfrac{f_i}{p}(w)} + \frac{ T_{\bar{z}_1}\frac{g}{p}(z) \overline{ T_{\bar{z}_1} \frac{g}{p}(w)}}{1-z_2 \bar{w}_2} \right)\\
 &   \ \ \ \ -  
 \sum_{i=1}^n P_{\theta} \left(\frac{-T_{\bar{z}_1}p(z) }{p(0,z_2)} \frac{f_i(z)}{p(z)} \right)  \overline{ T_{\bar{z}_1} \tfrac{f_i}{p}(w)}   +
 \frac{  \overline{T_{\bar{z}_1} \frac{g}{p}(z) T_{\bar{z}_1} \frac{g}{p}(w)}}{1-z_2 \bar{w}_2}
 \\
 & = \sum_{i=1}^n P_{\theta} \left(\frac{T_{\bar{z}_1}p(z) }{p(0,z_2)} \frac{f_i(z)}{p(z)} + T_{\bar{z}_1} \tfrac{f_i}{p}(z) \right)\overline{ T_{\bar{z}_1} \tfrac{f_i}{p}(w)}.
 \end{aligned} 
 \]
 Since 
 \[  T_{\bar{z}_1} \frac{f_i}{p}(z) = \frac{-T_{\bar{z}_1}p(z)}{p(0,z_2)} \frac{f_i(z)}{p(z)} + \frac{ T_{\bar{z}_1}f_i(z)}{p(0,z_2)}, \]
 we can simplify the main equation to 
 \[  \left[ S^*_{z_1}, S_{z_1} \right] K^2_w(z) = 
 \sum_{i=1}^n P_{\theta}  \left(\frac{ T_{\bar{z}_1}f_i(z)}{p(0,z_2)} \right)\overline{ T_{\bar{z}_1} \tfrac{f_i}{p}(w)}.  \] 
 Combining this with the result for $K^1_w$ gives: 
 \[ \left[ S^*_{z_1}, S_{z_1} \right]  K_w(z) =  P_{\theta}  \left(  \frac{1}{p(0,z_2)}\sum_{i=1}^n T_{\bar{z}_1}f_i(z) \overline{ T_{\bar{z}_1} \tfrac{f_i}{p}(w)} +f_i(0,z_2)\frac{ \overline{f_i(w)}}{\overline{p(w)}} \right).\]
 Recall that $\deg f_i \le (1, n-1)$. Thus, $\deg T_{\bar{z}_1}f_i \le (0,n-1)$ and  $\deg T_{\bar{z}_1}f_i(0,z_2) \le (0,n-1).$ This means that the set of linear combinations of functions of the form 
 \[ \sum_{i=1}^n T_{\bar{z}_1}f_i(z) \overline{ T_{\bar{z}_1} \tfrac{f_i}{p}(w)} -f_i(0,z_2)\frac{ \overline{f_i(w)}}{\overline{p(w)}}\]
 has at most dimension $n$. Thus, the set of linear combinations of the functions 
  \[  \left[ S^*_{z_1}, S_{z_1} \right]  K_w(z)  \]
  has at most dimension $n$. Since linear combinations of the reproducing kernel functions are dense in $\mathcal{K}_{\theta}$, this implies that the rank of $[S^*_{z_1}, S_{z_1}]$ is at most $n$.
 \end{proof}
%  \begin{remark}
%It is not clear if the rank $n$ in this theorem is always sharp. When $\theta(z) = \phi(z_1) \psi(z_2)$, as in Section \ref{s-exa}, then rank $[S_{z_1}, S_{z_1}^*] = n$ because $[ S_{z_1}, S_{z_1}^*]|_{\Smin_2} =0.$ However, if $\theta$ is not a product of two one-variable inner functions, these computations become substantially more complicated. In the Example \ref{ex:rational} (below) we perform exact computations for such a $(1,1)$ rational inner function. For for this example, we also observe that rank $[S_{z_1}, S_{z_1}^*] = 1.$ This leads us to conjecture the following:
%\begin{center}
%\textbf{Conjecture:} If $\phi$ is a rational inner function of degree $(1,n),$ then rank  $[S_{z_1}, S_{z_1}^*] = n.$ Namely, we cannot have rank  $[S_{z_1}, S_{z_1}^*] < n.$
%\end{center}
%\end{remark}

Conversely, using previous results about Agler kernels, it is not difficult to generalize the arguments from Section \ref{s-exa}
to conclude the $(\Rightarrow)$-statement. Again, we record this result with the following auxiliary theorem:
 
 \begin{theorem}\label{thm:FiniteRank} Let $\theta$ be an inner function on $\mathbb{D}^2$. Then, if $[S^*_{z_1}, S_{z_1}]$ has rank $n$ on $\mathcal{K}_{\theta},$ then $\theta$ is a rational inner function of degree less than or equal to $(1,n)$. 
 \end{theorem}

 \begin{proof}%[Proof of Theorem \ref{thm:FiniteRank}]
 Assume $[ S^*_{z_1}, S_{z_1}]$ has rank $n$ on $\mathcal{K}_{\theta}.$ We will show that 
 \[ \dim \mathcal{H}(K^{max}_1) \le n \text{ and } \dim \mathcal{H}(K^{min}_2)  \le 1.\] 
 Then Theorem \ref{thm:rational3} will imply that $\theta$ is a rational inner function of degree at most $(1,n).$
 
For the remainder of the proof, we establish these dimension bounds. First observe that if $f \in \Smax_1$, then 
 \begin{align}\label{e-commproj}
 \left[ S^*_{z_1}, S_{z_1} \right] f = f - P_{\theta}z_1 T_{z_1} f   = P_{\theta} f(0,z_2). 
 \end{align}
 Proceeding towards a contradiction, assume $\dim \mathcal{H}(K^{max}_1) >n.$ Then, we can find a function $f \in  \mathcal{H}(K^{max}_1)$ such that $\| f\|^2_{H^2} = 1$ and $f \in \ker [S^*_{z_1}, S_{z_1}]$.  This implies
 $f(0,z_2) \in \theta H^2(\mathbb{D}^2)$ and so
 \[ f(z) = z_1 T_{\bar{z}_1} f(z) + f(0,z_2) =  z_1 T_{\bar{z}_1} f(z) + \theta(z) g(z), \]
for some $g \in H^2(\mathbb{D}^2)$. But, by orthogonality 
\[ \| f(0,z_2) \|^2_{H^2} = \left \langle f(z),  f(0,z_2) \right \rangle_{H^2} =  
\left \langle f(z), \theta(z) g(z) \right \rangle_{H^2} =0,\]
and so $f(z) = z_1 T_{\bar{z}_1} f(z).$  Since $T_{\bar{z}_1}f$ is in $\mathcal{K}_{\theta}$ and $z_1T_{\bar{z}_1}f = f$ is in $\mathcal{K}_{\theta}$, we must have $T_{\bar{z}_1}f$ in $\Smax_1$ and $f \in z_1 \Smax_1.$ Since $f  \in \mathcal{H}(K^{max}_1)$, this means $f \perp z_1 \Smax_1$ and since $f$ is orthogonal to itself, $f \equiv 0,$ which gives the contradiction.

Now consider $\mathcal{H}(K^{min}_2).$ If $\dim \mathcal{H}(K^{min}_2) = 0$, we are done. If $\dim \mathcal{H}(K^{min}_2) \ge 1$, we can pick a normalized function $f\in \mathcal{H}(K^{min}_2).$ Then $\{ z^m_2 f \}_{m=0}^{n}$ is an orthonormal set in $\Smin_2.$ Since $\left[ S^*_{z_1}, S_{z_1} \right]$ has rank $n$, one can find a polynomial $p$ of degree at most $n$ such that
\begin{align}
0 &= \left[ S^*_{z_1}, S_{z_1} \right] p(z_2)f(z) \nonumber \\
&= p(z_2)f(0,z_2) 
-\left \langle p(w_2)f(w),  \frac{T_{\bar{z}_1} \theta(w) }{1-\bar{z}_2 w_2} \right \rangle_{H^2} T_{\bar{z}_1} \theta(z) +
\left \langle p(w_2)T_{\bar{z}_1}f(w), \frac{T_{\bar{z}_1} \theta(w) }{1-\bar{z}_2 w_2}\right \rangle_{H^2} \theta(z) \nonumber,
\end{align}
where the last line used Lemma \ref{lem:ShiftForm}. Now, observe that since $f \perp z_2 \Smin_2$, we can compute 
\[
\begin{aligned}
& \left \langle p(w_2)f(w),  \frac{T_{\bar{z}_1} \theta(w) }{1-\bar{z}_2 w_2} \right \rangle_{H^2} \\
& \qquad  = 
\left \langle p(w_2)f(w), w_2^{n+1} \bar{z}_2^{n+1} \frac{T_{\bar{z}_1} \theta(w) }{1-\bar{z}_2 w_2} \right \rangle_{H^2} +  \sum_{k=0}^{n} z_2^k \left \langle p(w_2)f(w),  w_2^k T_{\bar{z}_1} \theta(w)  \right \rangle_{H^2}\\
& \qquad = \sum_{k=0}^{n} z_2^k \left \langle p(w_2)f(w),  w_2^kT_{\bar{z}_1} \theta(w)  \right \rangle_{H^2} \\
& \qquad = -q(z_2), 
\end{aligned}
\]
where $q$ is a polynomial $q(z_2)$ with $\deg q$ at most $n.$ Similarly we can compute:
\[ 
\begin{aligned}
\left \langle p(w_2)T_{\bar{z}_1}f(w), \frac{T_{\bar{z}_1} \theta(w) }{1-\bar{z}_2 w_2}\right \rangle_{H^2} \theta(z) & = \left \langle p(w_2)f(w), \frac{w_1T_{\bar{z}_1} \theta(w) }{1-\bar{z}_2 w_2}\right \rangle_{H^2} \theta(z)  \\
& = - \left \langle p(w_2)f(w), \frac{ \theta(0,w_2) }{1-\bar{z}_2 w_2}\right \rangle_{H^2} \theta(z)  \\
& = - \left \langle p(w_2)f(0, w_2), \frac{ \theta(w) }{1-\bar{z}_2 w_2}\right \rangle_{H^2} \theta(z) \\
& = - P_{\theta H^2_2(\mathbb{D})} p(z_2)f(0,z_2). 
\end{aligned}
\]
Substituting those computations back into our previous equation gives:
\[0 = \left[ S^*_{z_1}, S_{z_1} \right] p(z_2)f(z)=  p(z_2)f(0,z_2) + q(z_2) T_{\bar{z}_1} \theta(z) +h(z_2)\theta(z)\label{e-constant}\]
for some $h \in H_2^2(\mathbb{D}).$ Then $\theta$ is rational of degree one in $z_1$ since solving the above equation for $T_{\bar{z}_1} \theta$ and substituting that into the following equation for $\theta$ gives:
\begin{align*}
\theta(z) &= z_1 T_{\bar{z}_1} \theta(z) + \theta(0,z_2) \\
&= z_1 \frac{ - p(z_2)f(0,z_2)- \theta(0,z_2)h(z_2)}{q(z_2) +z_1h(z_2)} + \theta(0,z_2) \\
&= \frac{ -z_1 p(z_2)f(0,z_2) + \theta(0,z_2)q(z_2)  }{q(z_2) +z_1h(z_2)}.
\end{align*}
Since $p, f(0,\cdot), \theta(0, \cdot), q,$ and $ h$ are in $H_2^2(\mathbb{D}),$ they all have non-tangential boundary limits at a.e.~$t\in \mathbb{T}$. Furthermore, it is easy to show that the function
\[ \Psi(z_1,t) \equiv \frac{ -z_1 p(t)f(0,t) + \theta(0,t)q(t)  }{q(t) +z_1h(t)}
\]
is bounded and analytic for almost every $t$. Indeed a simple computation shows that a zero in the denominator implies that the function must be constant in $z_1$.  Now, for almost every $t$ in $\mathbb{T}$, recall that $\theta(z_1, t)$ denotes the unique $H^2(\mathbb{D})$ function whose boundary values are $\theta(t_1, t)$ for a.e.~$t_1 \in \mathbb{T}.$  Observe that $\Phi(z_1,t)$ and $\theta(z_1,t)$ have the same boundary values a.e. Namely:
\[ 
\begin{aligned}
\lim_{r \nearrow 1} \theta(rt_1, t) &= \lim_{r \nearrow 1} \frac{ -r t_1 p(t)f(0,rt) + \theta(0,rt)q(rt)  }{q(rt) +rt_1h(rt)} \\
&= \lim_{r \nearrow 1} \frac{ t_1 p(t)f(0,t) + \theta(0,t)q(t)  }{q(t) +rt_1h(t)} \\
&= \lim_{r \nearrow 1} \Phi(z_1,t). 
\end{aligned}
\] 
Therefore, for a.e.~$t\in \mathbb{T}$, these one variable functions must agree:
\[
\theta(z_1, t) =  \Psi(z_1,t) = \frac{ -z_1 p(t)f(0,t) + \theta(0,t)q(t)  }{q(t) +z_1h(t)}.
\]
It follows that $\theta(z_1, t)$ has degree $1$ in $z_1.$
This means $H^2(\mathbb{T}) \ominus \theta(\cdot, t) H^2(\mathbb{T})$ is at most a one dimensional space, and so, Theorem \ref{thm:restriction} implies that $\mathcal{H}(K^{min}_2)$ is at most one dimensional. As mentioned already, the result follows immediately from the proven dimension bounds.
 \end{proof}
 
 To obtain Theorem  \ref{t-IFF}, we basically combine Theorems \ref{t-2n} and \ref{thm:FiniteRank} using several basic arguments:
 
 \begin{proof}[Proof of Theorem \ref{t-IFF}]
First, assume $\theta$ is rational inner of degree $(1,n)$. Then by Theorem \ref{t-2n}, $[S^*_{z_1}, S_{z_1}]$ has rank at most $n$. Proceeding to a contradiction, assume that the rank equals some $N$ strictly less than $n$. Then by Theorem \ref{thm:FiniteRank}, $\theta$ is rational inner of degree at most $(1,N)$, which is a contradiction. Thus, the rank of the commutator must be $n$. Similarly, if $\theta$ is rational inner of degree $(0,n)$, the same argument shows that rank$[S^*_{z_1}, S_{z_1}] =n.$

 Conversely, assume that the commutator has rank $n$ on $\mathcal{K}_{\theta}$, for some inner function $\theta$. Then by Theorem \ref{thm:FiniteRank}, $\theta$ is rational inner of degree at most $(1,n)$. Proceeding towards a contradiction, assume that the degree of $\theta$ is at most $(1,N),$ where $N$ is strictly less that $n$. Then by Theorem \ref{t-2n}, the rank of the commutator is strictly less than $n$, which is a contradiction. Thus, the degree of $\theta$ must be either $(0,n)$ or $(1,n)$.
 \end{proof}
 
% \begin{remark} \textnormal{It is worth noting that in Section \ref{s-exa}, we consider the operator $[S^*_{z_1}, S_{z_1}]$ while in the proofs of Theorems \ref{t-2n} and \ref{thm:FiniteRank}, we consider $[S^*_{z_1}, S_{z_1}].$ We adopt this convention on purpose because in the two separate situations, the chosen commutator yields the simplest formula. } 
% \end{remark}
 
%\begin{questions}
\subsection{Open questions} There are many related questions yet to be answered. Here, we describe two:

(a) In the product case, namely when $\theta(z) = \phi(z_1) \psi(z_2)$, Corollary \ref{c-essnorm} states that the essential normality of $S_{z_1}$ implies $[S_{z_1}^*, S_{z_1}]$ has finite rank. The results in this section do not allow us to conclude this in general. Thus, the following question remains:
\begin{center}
Does essential normality of $S_{z_1}$ imply that $[S_{z_1}^*, S_{z_1}]$ is finite rank?
\end{center}
This is a particularly intriguing question in light of the Guo-Wang result \cite{gw09}, which says that the joint essential normality of $S_{z_1}$ and $S_{z_2}$ does imply finite rank.

(b) It also seems difficult to extend Proposition \ref{p-spectrum} from the product case to the general case. Concerning part (c), even the spectrum of the commutator on $\Smax_1$ is interesting. In view of equation \eqref{e-commproj}, one can embed this question into a larger framework by interpreting the map $f\mapsto f(0,z_2)$ as a projection. It was suggested to us in private communications with R.G. Douglas that one should then ask the very general, rather attractive question: \begin{center}
For which $\theta, \vartheta$ is the projection $P_\theta P_\vartheta$ finite rank?
\end{center}
Alternatively, it may be possible to generalize Proposition \ref{p-spectrum} to rational inner functions $\theta$, as they are typically more tractable.
%\end{questions}

\section{Reducing Agler Subspaces for $S_{z_1}$  --- The Proof of Theorem \ref{t-iff}}\label{s-reducing}
 When $\theta$ is a product of one variable inner functions, namely $\theta(z)= \phi(z_1)\psi(z_2)$, then the compressed shifts $S_{z_1}$ and $S_{z_2}$ have simple reducing subspaces. Namely, in Subsection \ref{ss-prodreducing}, we proved that $\Smax_1$ and $\Smin_2$ are reducing for $S_{z_1},$ and $\Smax_2$ and $\Smin_1$ are reducing for $S_{z_2}.$ 
 
 This motivates the following question. Let $\theta$ be an arbitrary inner function and let $(K_1, K_2)$ be Agler kernels of $\theta.$ Then:
 \begin{center} \emph{ When are 
$\mathcal{H} \left(\frac{K_1(z,w)}{1-z_1\bar{w}_1} \right)$ and $\mathcal{H} \left(\frac{K_2(z,w)}{1-z_2\bar{w}_2} \right)$
 reducing subspaces for $S_{z_1}$?} \end{center}
 
 One should note that this question only considers when Agler subspaces are reducing subspaces. Indeed, a characterization of the inner functions $\theta$ for which $S_{z_1}$ and/or $S_{z_2}$ have reducing subspaces on the model space $\mathcal{K}_{\theta}$ seems difficult with the techniques at hand.
 
 \subsection{Reducing Subspaces and Agler Kernels}

It is actually easy to characterize when Agler kernels $(K_1,K_2)$ are associated to reducing subspaces. The result is as follows:

 \begin{theorem}\label{thm:Reducing1} Let $(K_1, K_2)$ be Agler kernels of $\theta.$ Then the Agler subspaces
 \[ \mathcal{S}_1 \equiv \mathcal{H} \left(\frac{K_1(z,w)}{1-z_1\bar{w}_1} \right) \ \text{ and } \ \mathcal{S}_2 \equiv\mathcal{H} \left(\frac{K_2(z,w)}{1-z_2\bar{w}_2} \right)  \]
 are reducing subspaces for $S_{z_1}$ if and only if they are subspaces of $\mathcal{K}_{\theta}$ and $K_1(z,w)$ is a function of only $z_2$ and $\bar{w}_2$.
 The analogous statement holds for $S_{z_2}.$
 \end{theorem} 

 \begin{proof} 
($\Leftarrow$)   First, observe that since $S_{z_1} \left( \mathcal{S}_1 \right) \subseteq \mathcal{S}_1$, then $\mathcal{S}_1$ and $\mathcal{S}_2$ are reducing subspaces for $S_{z_1}$ if and only if $P_{\mathcal{S}_1} S_{z_1}|_{\mathcal{S}_2}$ is the zero operator, which occurs if and only if $P_{\mathcal{S}_2} S^*_{z_1}|_{\mathcal{S}_1}$ is the zero operator. 

 Now, assume  $K_1(z,w)$ is a function of only $z_2$ and $\bar{w}_2.$ Fix $w\in \mathbb{D}^2$. Then
 \[ S_{z_1}^* \left( \frac{K_1(\cdot, w)}{1- \cdot  \ \bar{w}_1} \right)(z) =  \frac{\bar{w}_1 K_1(z,w)}{1-z_1 \bar{w}_1} \in \mathcal{S}_1. \]
 Since linear combinations of these reproducing kernel functions are dense in $\mathcal{S}_1$, this implies that $S_{z_1}^* \left (\mathcal{S}_1 \right) \subseteq \mathcal{S}_1$ and so $P_{\mathcal{S}_2} S_{z_1}^*|_{\mathcal{S}_{1}}$ is the zero operator, which means $\mathcal{S}_1$ and $\mathcal{S}_2$ are reducing.

($\Rightarrow$) Conversely, assume $\mathcal{S}_1$ and $\mathcal{S}_2$ are reducing subspaces. By manipulating \eqref{eqn:akernels}, one can obtain:
 \[ K_1(z,w) = \frac{z_1 \bar{w}_1 K_2(z,w)}{1-z_2\bar{w}_2}  + \frac{1}{1-z_2 \bar{w}_2} -  \frac{K_2(z,w)}{1-z_2\bar{w}_2} -   \frac{ \theta(z) \overline{\theta(w)}}{1-z_2 \bar{w_2}}. \]
 Fix $w \in \mathbb{D}^2$ and apply $S_{z_1}^*$ to obtain:
 \[  S_{z_1}^* \left( K_1(\cdot ,w) \right)(z) = \frac{ \bar{w}_1 K_2(z,w)}{1-z_2\bar{w}_2} -  \frac{ S_{z_1}^*\left(K_2(\cdot, w) \right) (z)}{1-z_2\bar{w}_2}  -   \frac{ S_{z_1}^*\theta(z) \overline{\theta(w)}}{1-z_2 \bar{w_2}}\,. \] 
As was shown in Propositions 3.4 and 3.5 in \cite{b12}, for each $w \in \mathbb{D}^2$
 \[ \frac{ S_{z_1}^*\left(K_2(\cdot, w) \right) (z)}{1-z_2\bar{w}_2}, \ \  \frac{ S_{z_1}^*\theta(z) \overline{\theta(w)}}{1-z_2 \bar{w_2}} \in \mathcal{S}_2. \]
Since $S_{z_1}^* \left( K_1(\cdot ,w) \right)$ is  a sum of functions in $\mathcal{S}_2$, it must be in $\mathcal{S}_2$ as well.  By assumption, since $\mathcal{S}_1$ and $\mathcal{S}_2$ are reducing, $P_{\mathcal{S}_2} S^*_{z_1}|_{\mathcal{S}_1}$ is the zero operator and so we must have
 \[ S_{z_1}^* \left( K_1(\cdot ,w) \right) \equiv 0, \]
This means that each $K_1(\cdot, w)$ is a function of only $z_2.$  As linear combinations of these functions are dense in $\mathcal{H}(K_1)$, all functions in $\mathcal{H}(K_1)$ are functions of only $z_2$. Now let $\{f_i\}$ be an orthonormal basis of $\mathcal{H}(K_1)$. Then the reproducing kernel $K_1$ satisfies  $K_1(z,w) = \sum f_i(z) \overline{f_i(w)},$ which implies $K_1(z,w)$ is a function of only $z_2$ and $\bar{w}_2.$  \end{proof}

\subsection{Products of Inner Functions and their Agler Kernels}

We would like to answer the question:~Which $\theta$ have Agler subspaces that are reducing for  $\mathcal{S}_1$ or $\mathcal{S}_2$? The previous result characterized the existence of such Agler reducing subspaces using properties of the Agler kernels. The following result provides the needed link between Theorem \ref{thm:Reducing1} and the question of interest, at least in the case when $\theta$ is rational inner.

\begin{theorem} \label{thm:Reducing2} Let $\theta$ be an inner function on $\mathbb{D}^2$. Then the following are equivalent:
\begin{itemize}
\item[(a)] The function $\theta$ is a product of of two one variable inner functions;
\item[(b)] There is some pair $(K_1,K_2)$ of Agler kernels of $\theta$ such that $K_1$ is a function of only $z_2$ and $\bar{w}_2$ and 
 \begin{equation} \label{eqn:lim} \lim_{r \nearrow 1} \left( 1- r^2 |\tau_1|^2 \right) K_2(r\tau_1, z_2, r\tau_1, w_2) = 0,\end{equation}
 for every $z_2, w_2 \in \mathbb{D}$ and almost every $\tau_1 \in \mathbb{T};$
 \item[(c)] There is some pair $(\widetilde K_1,\widetilde K_2)$ of Agler kernels of $\theta$ such that $\widetilde K_2$ is a function of only $z_1$ and $\bar{w}_1$ and 
 \begin{equation*} \lim_{r \nearrow 1} \left( 1- r^2 |\tau_2|^2 \right)\widetilde K_1(z_1, r\tau_2, w_1, r\tau_2) = 0,
 \end{equation*}
 for every $z_1, w_1 \in \mathbb{D}$ and almost every $\tau_2 \in \mathbb{T}.$
  \end{itemize}
 \end{theorem}
 \begin{proof} 
 We prove (a) $\Leftrightarrow$ (b). The equivalence of part (c) follows by symmetry.\\
 $(\Rightarrow)$ If $\theta(z) = \phi(z_1) \psi(z_2),$ then $(\Smax_1, \Smin_2)$ are reducing subspaces for $S_{z_1}$ and the calculations in Section \ref{s-exa} show that the associated Agler kernels are 
 \[ 
 \begin{aligned} K^{max}_1(z,w) &= \frac{1-\psi(z_2) \overline{\psi(w_2)}}{1-z_2 \bar{w}_2}; \\
 K^{min}_2(z,w) &= \frac{ \psi(z_2) \overline{\psi(w_2)} (1-\phi(z_1) \overline{\phi(w_1)})}{1-z_1 \bar{w}_1}\,.
 \end{aligned}
 \]
 Then, $K^{max}_1(z,w)$ is clearly a function of only $z_2$ and $\bar{w}_2$. Moreover, since $\phi$ is an inner function,
 \[ \lim_{r \nearrow 1} \left( 1- r^2 |\tau_1|^2 \right) K^{min}_2(r\tau_1, z_2, r\tau_1, w_2) =  \lim_{r \nearrow 1} \left( 1- r^2 |\tau_1|^2 \right) \frac{  \psi(z_2) \overline{\psi(w_2)} ( 1-| \phi( r \tau_1)|^2)}{1-r^2 |\tau_1|^2} =0 \]
 for every $z_1$ and $w_1$ and almost every $\tau_1 \in \mathbb{T}.$ 
 \\
 $(\Leftarrow$) Now assume $K_1(z,w)$ is a function of only $z_2$ and $\bar{w}_2$ and $K_2(z,w)$ satisfies \eqref{eqn:lim}. Abusing notation slightly, we write $K_1(z_2,w_2).$  We will construct inner functions $\phi$ and $\psi$ so that $\theta(z) = \phi(z_1) \psi(z_2).$ However, first we need several preliminary computations.\\
 
\textbf{Preliminary Computation 1:} We first show that there is a $w_2 \in \mathbb{D}$ such that $\lim_{r \nearrow 1} \theta(r \tau_1, w_2)$ exists and is nonzero for almost every $\tau_1$ in $\mathbb{T}.$  
 Fix an arbitrary $w_2\in \mathbb{D}$ and consider what happens when $\lim_{r \nearrow 1} \theta(r \tau_1, w_2) =0$ for some $\tau_1$ where \eqref{eqn:lim} holds. Then,
\[ 
\begin{aligned}
1 &= \lim_{r \nearrow 1} ( 1- |\theta(r \tau_1, w_2)|^2 ) \\
&= \lim_{r \nearrow 1} (1- r^2 |\tau_1|^2) K_2(r\tau_1, w_2, r\tau_1, w_2) + ( 1- |w_2|^2) K_1(w_2, w_2) \\
&= ( 1- |w_2|^2) K_1(w_2, w_2). 
\end{aligned}
\]
Rewriting \eqref{eqn:akernels} and setting $z_1 = w_1$ and $z_2 = w_2$ gives
\[
1 - |\theta(z_1,w_2)|^2 =  ( 1- |w_2|^2) K_1(w_2, w_2) + (1- |z_1|^2)K_2(z_1, w_2, z_1, w_2),
\]
and together we obtain
\[ | \theta(z_1,w_2)|^2  = -(1- |z_1|^2)K_2(z_1, w_2, z_1, w_2) \]
for every $z_1,$ which implies $\theta(\,\cdot\,,w_2)$ is identically zero. Therefore, if for every $w_2$, there is some $\tau_1$ satisfying \eqref{eqn:lim} such that $\lim_{r \nearrow 1} \theta(r \tau_1, w_2) = 0$, then $\theta \equiv 0$, which is a contradiction.  
Thus, there is some $w_2$ such that  $\lim_{r \nearrow 1} \theta(r \tau_1, w_2) \ne 0$ for any $\tau_1$ satisfying \eqref{eqn:lim}.  However, since,  $\theta(\cdot, w_2)$ is bounded and holomorphic, it is in $H_1^2(\mathbb{D})$ and so for that $w_2$
\begin{equation} \label{eqn:lim2} \lim_{r \nearrow 1} \theta(r \tau_1, w_2) \text{ exists and is nonzero} \end{equation} 
 for almost every $\tau_1.$ \\

\textbf{Preliminary Computation 2:}  Now, recall that, for almost every $\tau_1$ in $\mathbb{T}$, $\theta$ has boundary values at $(\tau_1, \tau_2)$ for almost every $\tau_2.$ Then, there is a well-defined inner function, which we call $\theta_{\tau_1}( z_2),$ that satisfies the following boundary conditions
 \[ \lim_{r \nearrow 1} \theta_{\tau_1}( r\tau_2) = \theta( \tau_1, \tau_2) \] 
 for almost every $\tau_2 \in \mathbb{T}.$ Fix such a $\tau_1$ and further, assume $\theta$ satisfies limit conditions  \eqref{eqn:lim} and \eqref{eqn:lim2} with $\tau_1.$  We will show that 
  \[ \Phi_{\tau_1}( z_2) \equiv \lim_{r \nearrow 1} \theta(r \tau_1, z_2) = \theta_{\tau_1}(z_2). \]
Consider the $w_2$ found earlier. Now we can use \eqref{eqn:lim}, \eqref{eqn:lim2}, and then \eqref{eqn:akernels}  with $z_1 = r\tau_1$ and $w_1 = r\tau_1$ to write:
\[ \Phi_{\tau_1}(z_2)= \lim_{r \nearrow 1 } \theta( r \tau_1, z_2) = \frac{1}{\overline{\Phi_{\tau_1}( w_2)}} \left( 1 - (1-z_2 \bar{w}_2) K_1(z_2,w_2) \right). \]
This shows that $\Phi_{\tau_1}(z_2)$ is a well-defined function in $H_2^2(\mathbb{D}).$ Moreover, rewriting \eqref{eqn:akernels}, one can obtain:
\[ \lim_{r \nearrow 1} \theta( r \tau_1, r \tau_2) = \frac{1}{ \overline{ \Phi_{\tau_1}(w_1)}}  \lim_{r \nearrow 1} \left( 1 - (1- r^2| \tau_1|^2) K_2(r\tau_1, r \tau_2, r\tau_1, w_2) - (1-r \tau_2 \bar{w}_2)
K_1(r\tau_2, w_2) \right). \]
Then, we can use Cauchy-Schwarz to calculate:
\[
 \begin{aligned}
 \lim_{r \nearrow 1} &\left( (1- r^2| \tau_1|^2) K_2(r\tau_1, r \tau_2, r\tau_1, w_2)  \right) \le  \lim_{r \nearrow 1}\left( ( 1- r^2| \tau_1|^2) K_2(r\tau_1, w_2, r\tau_1, w_2) \right)^{1/2} \\
\hspace{1in}  &  \cdot \left(  (1- r^2| \tau_1|^2) K_2(r\tau_1, r\tau_2, r\tau_1, r\tau_2) \right)^{1/2} =0,
\end{aligned}
\]
since the first term is tending to zero by \eqref{eqn:lim} and the second term is bounded for almost every $\tau_2$. This implies that for almost every $\tau_2$ in $\mathbb{T}$, 
\[  \theta(\tau_1, \tau_2) = \lim_{r \nearrow 1} \theta( r \tau_1, r \tau_2) = \frac{1}{ \overline{ \Phi_{\tau_1}(w_2)}} \left( 1-  \lim_{r \nearrow 1} (1-r \tau_2 \bar{w}_2)
K_1(r\tau_2, w_2) \right) = \lim_{r \nearrow 1}  \Phi_{\tau_1}( r \tau_2). \]
Then, since $\Phi_{\tau_1}$ and $\theta_{\tau_1}$ are both $H_2^2(\mathbb{D})$ with the same boundary values, they must be equal. \\

\textbf{Construction of $\phi$ and $\psi$:} Now, fix any $\mu_1$ in $\mathbb{T}$ that satisfies the limit conditions \eqref{eqn:lim} and \eqref{eqn:lim2} with $\theta_{\mu_1}(z_2)$ well-defined and inner. Then, we can conclude that for every $\tau_1$ satisfying the inner function and limit conditions (by assumption, this is almost every $\tau_1$), 
\[ \theta_{\tau_1}(z_2) \overline{ \theta_{\tau_1}(w_2)} = 1 - (1-z_2 \bar{w}_2) K_1(z_2,w_2) = 
 \theta_{\mu_1}(z_2) \overline{ \theta_{\mu_1}(w_2)} \]
and so, taking boundary limits, for almost every $\tau_2$ in $\mathbb{T}$, 
\[ \theta(\tau_1, \tau_2)   =\theta(\mu_1, \tau_2)  \frac{\overline{ \theta_{\mu_1}(w_2)} }{\overline{ \theta_{\tau_1}(w_2)}}.  \]
Since $\theta_{\tau_1}$ and $\theta_{\mu_1}$ are both inner, this implies that 
\[  \left | \frac{\overline{ \theta_{\mu_1}(w_2)} }{\overline{ \theta_{\tau_1}(w_2)}} \right | =1, \]
and so we can obtain:
\[  \theta(\tau_1, \tau_2)  =\theta(\mu_1, \tau_2)  \frac{\theta_{\tau_1}(w_2) }{ \theta_{\mu_1}(w_2)} . \]
Define $\psi(z_2) = \theta_{\mu_1}(z_2)$ and $\phi(z_1) =  \frac{ \theta(z_1, w_2)}{\theta_{\mu_1}(w_2)}.$
 Then, $\psi(z_2)\phi(z_1)$ is a product of one variable inner functions and for almost every $\tau_1, \tau_2$, we have
\[ \lim_{r \nearrow 1}\psi(r \tau_2)\phi( r \tau_1) = \lim_{r \nearrow 1}  \theta_{\mu_1}( r\tau_2)\frac{ \theta(r \tau_1, w_2)}{\theta_{\mu_1}(w_2)} =\theta(\mu_1, \tau_2)  \frac{\theta_{\tau_1}(w_2) }{ \theta_{\mu_1}(w_2)} =  \theta(\tau_1, \tau_2),\]
 which implies that $\theta(z_1,z_2) = \phi(z_1) \psi(z_2).$
 \end{proof}

\subsection{Proof of Theorem \ref{t-iff}}
We combine Theorems \ref{thm:Reducing1} and \ref{thm:Reducing2} to prove Theorem \ref{t-iff}, which we state again for the convenience of the reader.

\noindent \textbf{Theorem \ref{t-iff}.} \emph{Let $\theta$ be a rational inner function on $\mathbb{D}^2.$ Then $\theta$ has a pair of Agler kernels $(K_1, K_2)$ such that the associated Agler spaces
 \[ \mathcal{S}_1 \equiv \mathcal{H} \left(\frac{K_1(z,w)}{1-z_1\bar{w}_1} \right) \ \text{ and } \ \mathcal{S}_2 \equiv\mathcal{H} \left(\frac{K_2(z,w)}{1-z_2\bar{w}_2} \right)  \]
 are reducing subspaces for $S_{z_1}$ if and only if $\theta$ is a product of two one variable inner functions. By symmetry, this occurs if and only if  $\theta$ has a pair of Agler kernels $(\widetilde{K}_1, \widetilde{K}_2)$ such that the associated Agler spaces are reducing subspaces for $S_{z_2}$.}

\begin{proof}[Proof of Theorem \ref{t-iff}]
$(\Leftarrow)$ If $\theta$ is a product of one variable inner functions, then $\theta$ has such reducing subspaces by Proposition \ref{p-representation}. 

$(\Rightarrow)$ Assume $\theta$ is rational inner and possesses reducing subspaces of $S_{z_1}$ with kernels $(K_1,K_2)$. Then by Theorem \ref{thm:Reducing1}, $K_1(z,w)$ is a function of only $z_2$ and $\bar{w}_2.$ Now, we use the structure of rational inner functions to show that $K_2$ satisfies the limit condition \eqref{eqn:lim}. Indeed, by Lemma 9.1 in \cite{bk13}, we can write $\theta = \tilde{p} /p$, where $p$ has only finitely many zeros on $\mathbb{T}^2.$ By Theorem 2.8 in \cite{knese10}, we can write
\[ K_2(z,w) = \sum_{i=1}^N f_i(z) \overline{f_i(w)} = \sum_{i=1}^N \frac{ q_i(z)}{p(z)} \frac{ \overline{q_i(w)}}{\overline{p(w)}}, \]
where the $q_i$ are polynomials and $N$ only depends on $\deg \theta.$ Fix any $f_i$ as above and $\tau_1$ in $\mathbb{T}$. Then by Proposition 4.9.1 in \cite{Rud69}, if $f_i(z)$ has a singular point at $(\tau_1, z_2)$ for any $z_2$ in $\mathbb{D}$, then $f_i(z)$ has a singular point at $(\tau_1, \tau_2)$ for every $\tau_2$ in $\mathbb{T}.$ However, this cannot happen, as  the singular points of $f_i$ must occur at the zeros of $p$ and $p$ has only finitely many zeros on $\mathbb{T}^2.$ Thus, every $(\tau_1, z_2)$ must be a regular point of $f_i$, namely, $f_i$ must extend analytically to a neighborhood of $(\tau_1, z_2).$  Specifically, this means that for each fixed $z_2$ in $\mathbb{D},$
\[ c_i(\tau_1, z_2 ) \equiv  \lim_{r \nearrow 1} f_i( r \tau_1, z_2) \]
exists. Thus,
\[ \lim_{r \nearrow 1} K_2(r \tau_1 z_2, r\tau_1, w_2) = \sum_{i=1}^N c_i(\tau_1, z_2) \overline{c_i(\tau_1, w_2)}. \]
This implies that
\[ \lim_{r \nearrow 1} (1-|r \tau_1|^2) K_2( r \tau_1, z_2, r\tau_1, w_2) = \lim_{r \nearrow 1} (1-|r \tau_1|^2)\sum_{i=1}^N c_i(\tau_1, z_2) \overline{c_i(\tau_1, w_2)} =0,\]
for all $z_2, w_2,$ and $\tau_1.$  Thus, we can apply Theorem \ref{thm:Reducing2} to conclude that $\theta$ must be a product of one variable inner functions.
\end{proof}
\subsection{Open questions} There are also many related questions yet to be answered about reducing subspaces.  For example:\\
(a) Theorem \ref{t-iff} only characterizes reducing subspaces if $\theta$ is rational inner. This leads to the question:~Does Theorem \ref{t-iff} generalize to arbitrary inner functions? Namely,
\begin{center}
Do $S_{z_1}$ or $S_{z_2}$ have reducing Agler subspaces if and only if $\theta$ is a product of two one variable inner functions?
\end{center}
(b) Unfortunately, our tools only allowed us to answer questions about when Agler subspaces are reducing. This leaves open the difficult but attractive question:
\begin{center} For which $\theta,$ does $S_{z_1}$ or $S_{z_2}$ have reducing subspaces in $\mathcal{K}_{\theta}?$ \end{center}

\end{document}